\documentclass[12pt]{amsart}  
\usepackage{amssymb,latexsym}
\usepackage{amssymb,latexsym}
\usepackage{amsfonts,mathrsfs}
\usepackage[margin=1.4in]{geometry}
\usepackage{fancyhdr}
\usepackage{hyperref}
\usepackage{enumitem}
\usepackage{xcolor}
\usepackage[all]{xy}

\newtheorem{theorem}{Theorem}[section]
\newtheorem{lemma}[theorem]{Lemma}
\newtheorem{proposition}[theorem]{Proposition}
\newtheorem{corollary}[theorem]{Corollary}
\newtheorem{definition}[theorem]{Definition}
\theoremstyle{definition}
\newtheorem{example}[theorem]{Example}
\newtheorem*{conjecture}{Conjecture}
\theoremstyle{remark}
\newtheorem{remark}{Remark}[section]

\def\XXint#1#2#3{{\setbox0=\hbox{$#1{#2#3}{\int}$ }
\vcenter{\hbox{$#2#3$ }}\kern-.6\wd0}}

\newcommand{\End}{\mathrm{End}}

\newcommand{\Hom}{\mathrm{Hom}}
\newcommand{\HC}{\mathrm{HC}}
\newcommand{\Aut}{\mathrm{Aut}}
\newcommand{\Res}{\mathrm{Res}}
\newcommand{\Sym}{\mathrm{Sym}}
\newcommand{\Ind}{\mathrm{Ind}}
\newcommand{\sgn}{\mathrm{sgn}}

\newcommand{\IH}{\mathrm{IH}}
\newcommand{\Q}{\mathbb{Q}}
\newcommand{\R}{\mathbb{R}}
\newcommand{\C}{\mathbb{C}}
\newcommand{\V}{\mathbb{V}}
\newcommand{\Z}{\mathbb{Z}}
\newcommand{\A}{\mathbb{A}}
\newcommand{\fg}{\mathfrak{g}}

\def\QQ{\mathbb{Q}}

\usepackage{lipsum}                     
\usepackage{xargs}                      
\usepackage[colorinlistoftodos,prependcaption,textsize=tiny]{todonotes}

\newtheorem*{hzconj}{\textup{\bf Harris-Zucker Conjecture}}
\makeatletter

\newcommand{\Rmnum}[1]{\expandafter\@slowromancap\romannumeral #1@}
\makeatother
\begin{document}
\allowdisplaybreaks
\title{Hodge Classes in the Cohomology of Local Systems}

\pagestyle{fancy}
\fancyhf{}
\renewcommand{\headrulewidth}{0pt}
\fancyhead[CE]{}
\fancyhead[CO]{\leftmark}
\fancyhead[LE,RO]{\thepage}

\author{Xiaojiang Cheng}
\email{xiaojiangcheng@wustl.edu}
\address{Department of Mathematics, Washington University in St. Louis, 1 Brookings Drive, Campus Box 1146, St. Louis, MO 63130-4899}
\begin{abstract}
We study the Hodge conjecture for certain families of varieties over arithmetic quotients of balls and Siegel domain of degree two. As a byproduct, we derive formulas for Hodge numbers in terms of automorphic forms.
\end{abstract}

\maketitle
\numberwithin{equation}{section}

\section{Introduction}

The Hodge conjecture is one of the most important open problems in mathematics.  For a complex projective $n$-manifold, it says that \emph{Hodge $(p,p)$-classes} defined by topological and analytic conditions should be representable as fundamental classes of algebraic cycles of codimension $p$.  The only case known in general is the Lefschetz theorem on $(1,1)$-classes (and its dual, for $p=n-1$).  Even for abelian varieties, the Hodge conjecture is still a wide-open question, notwithstanding important deep results of Deligne \cite{DM} and others \cite{Ma, Mu, Sc}.

However, there are interesting \emph{strategies} for attacking the Hodge conjecture (HC) in general.  Inspired by Lefschetz's original approach, one idea features normal functions arising from fibering a Hodge class out over a base \cite{BFNP, GG, KP}. Philosophically, it is also natural to break the HC into a question about the absoluteness of Hodge classes (or field of definition of Hodge loci) and the HC for varieties defined over $\bar{\QQ}$ \cite{Vo}.

Universal families over locally symmetric spaces (connected Shimura varieties) provide a source of varieties defined over $\bar{\QQ}$ which come endowed with natural fibrations.
In this paper, we investigate what can be said about Hodge classes on such automorphic total spaces. The central point is that, under the Decomposition Theorem \cite{Sa, CM, KL}, these Hodge classes live in the intersection cohomology of automorphic local systems on the Shimura variety, which can be calculated using representation-theoretic tools.
The computation in \cite{Ar1} for the universal genus 2 curve over a Siegel modular threefold is one good example, and we want to extend Arapura's result (that the HC holds in this case) to some other interesting families.

So let $D=G/K$ be a Hermitian symmetric domain and $X=\Gamma\backslash D$ be an arithmetic quotient. We are interested in the cohomology group $H^*(X,\mathbb{V})$ where $\mathbb{V}$ is a local system underlying a homogeneous variation of Hodge structure on $X$. The groups $H^*(X,\mathbb{V})$ can be studied in terms of relative Lie algebra cohomology groups. Then the essential point is to study the decomposition of $L^2$ automorphic forms $\mathcal{A}^2(G,\Gamma)$, and use the Vogan-Zuckerman theory to determine the nonzero spaces $H^*(\mathfrak{g}, K; V\otimes U)$ and their Hodge numbers.  Here $U$ and $V$ are representations of $\mathfrak{g}=Lie(G)$, with $U\subset \mathcal{A}^2(G,\Gamma)$ unitary (infinite-dimensional or trivial) and $V$ a finite-dimensional Hodge representation.

The decomposition of $\mathcal{A}^2(G,\Gamma)$ is in general a very difficult problem and is only accessible for very simple groups $G$ and $\Gamma$.  However, when $G=U(3,1)$ we obtain the following ``unitary'' analogue of Arapura's result:\vspace{2mm}

\noindent \textbf{Theorem A} (= Thm.~\ref{thA}). Let $X$ be a compactification of the universal genus four Picard curve over an arithmetic quotient of $\mathbb{B}^3$.  Then the HC holds for $X$.\vspace{2mm}

\noindent We get a similar result for families of Picard-rank 16 $K3$ surfaces with cubic automorphism over an arithmetic quotient of the 2-ball (Theorem \ref{tK3}).

For $G=GSp(4)$, we use Arthur's classification of the automorphic representations and more recent work of Schmidt on Saito-Kurokawa lifts to establish the following\vspace{2mm}

\noindent \textbf{Propositon B} (= Prop.~\ref{thB}).  Let $\Gamma=\Gamma^{\mathrm{para}}(p)$ be a paramodular subgroup of prime level; then the multiplicity of the non-tempered representation $\sigma_k$ in $\mathcal{A}^2(Sp(4,\R),\Gamma)$ is given in terms of classical spaces of cusp forms:
   \begin{itemize}[leftmargin=0.4cm]
       \item $\dim S_{2k-2}(SL(2,\Z))+\dim S_{2k-2}(\Gamma_0(p))^{new,+}$ if $k$ is odd;
       \item $\dim S_{2k-2}(\Gamma_0(p))^{new,-}$ if $k$ is even.
\end{itemize}
\vspace{2mm}

\noindent This leads to our second result on the Hodge Conjecture:\vspace{2mm}

\noindent \textbf{Theorem C} (= Thm.~\ref{thC} and its Corollary).  The HC holds for the self-fiber product of the universal genus 2 curve, as well as for the universal abelian surface (and any compactification thereof), over $\Gamma^{\mathrm{para}}(p)\backslash \mathfrak{H}_2$ when $p=1,2,3,5$.
\vspace{2mm}

\noindent The meaning of the HC for a general quasi-projective variety is recalled in \S\ref{S2.4}.
{\begin{center} $*\;\;\;\;\;*\;\;\;\;\;*$\end{center}}

The paper is organized as follows. Section 2 is a summary of cohomology theories in proving the theorems. Section 3 consists of explicit computations of cohomological representations and their cohomology. Multiplicity formulas are given in section 4. And finally, we combine all results in section 5 to get applications to Hodge classes and the Hodge Conjecture.

\section*{Acknowledgements} The author would be pleased to express his heartfelt gratitude to Prof. Matt Kerr for introducing the interesting problem and for his innumerable discussions. He would also thank Dan Petersen, Hang Xue, Jayce Getz, and Ralf Schmidt for their kind help in automorphic representation theory. 


\section{Cohomology theories}

\subsection{Hodge representations}

The homogeneous variations of Hodge structure and underlying local systems which we will study over locally symmetric varieties arise from representations $V$ of a reductive group $G$ with additional data.

\subsubsection{Mumford-Tate groups and domains}

Fix an integer $n$. Let $(V,Q)$ be a pair consisting of a finite-dimensional $\Q$-vector space $V$ and a non-degenerate bilinear form $Q$ with $Q(u,v)=(-1)^nQ(v,u)$. A (polarized) Hodge structure of weight $n$ on $(V, Q)$ is  a decomposition $V_{\C}=\bigoplus\limits_{p+q=n}V^{p,q}$ satisfying the following conditions: 
\begin{itemize}[leftmargin=0.4cm]
    \item $\overline{V^{p,q}}=V^{q,p}$;
    \item $Q(u,v)=0$ for $u\in V^{p,q},v\in H^{p',q'},p\neq p'$;
    \item $i^{p-q}Q(u,\bar{u})>0$ for $0\neq u\in V^{p,q}$.

\end{itemize}

The Hodge numbers $h^{p,q}$ are just the dimensions of $V^{p,q}$. Equivalently, a Hodge structure can be considered as a real representation $\varphi: S^1\to \Aut(V_{\C}, Q)$ by setting $z\cdot v=z^{p-q}v$ for $z\in S^1,v\in V^{p,q}$ and extending by complex linearity.

Let $\mathcal{D}=\mathcal{D}_{\textbf{h}}=\mathcal{G}_{\R}/\mathcal{G}_{\R}^0$ be the period domain parameterizing $Q$-polarized Hodge structures on $V$ with Hodge numbers $\textbf{h}=(h^{n,0},\cdots,h^{0,n})$. Here $\mathcal{G}_{\mathbb{R}}=\Aut(V_{\R},Q)$ is either an orthogonal group $O(a,2b)$ (if $n$ is even) or a symplectic group $Sp(2r,\R)$ (if $n$ is odd), and $\mathcal{G}_{\R}^0$ is the compact stabilizer of a fixed $\varphi\in\mathcal{D}$. To each Hodge structure $\varphi\in\mathcal{D}$ is associated a $\Q$-algebraic Hodge group $\textbf{G}_{\varphi}\subset\Aut (V,Q)$, and a Mumford-Tate domain $D=D_{\varphi}=G_{\varphi}\cdot \varphi\subset \mathcal{D}$, where $G_{\varphi}=\textbf{G}_{\varphi}(\mathbb{R})$. The Hodge group $\textbf{G}_{\varphi}$ is defined to be the $\mathbb{Q}$-algebraic closure of $\varphi(S^1)$. It may also be defined as the stabilizer of the Hodge tensors of $\varphi$.

\begin{example}

    We consider a pair $(V, Q)$ where $V$ is a $\mathbb{Q}$-vector space of dimension six and $Q$ a nondegenerate alternating form. We assume there is an embedding $\mathbb{F}(:=\mathbb{Q}(\sqrt{-3})=\mathbb{Q}(\omega))\hookrightarrow \End_{\mathbb{Q}}(V)$. Then we have a decomposition $V_{\mathbb{F}}=V_+\oplus V_-$ into the  $\pm$ eigenspaces of the action of $\mathbb{F}$. We assume further that $V_{\pm}$ are isotropic with respect to $Q$ and the  Hermitian form on $V_{+,\mathbb{C}}$ $H(u,v)=\pm iQ(u,\overline{v})$ has signature $(2,1)$.

    Let $\varphi$ be a weight one Hodge structure over $(V, Q)$ commuting with the $\mathbb{F}$-action. So the Hodge decomposition $V_{\mathbb{C}}=V^{p,q}$ is compactible with the decomposition $V_{\mathbb{F}}=V_+\oplus V_-$. We assume that $\dim V^{1,0}_+=V^{0,1}_-=1$, $\dim V^{1,0}_-=V^{0,1}_+=2$. The Hodge structures of this type are precisely the Hodge structure associated with Picard curves.

     Let $G=\mathcal{U}(2,1):=\Aut_{\mathbb{Q}}(V,Q)\cap \mathrm{Res}_{\mathbb{F}/\mathbb{Q}}GL_{\mathbb{F}}(V)$. This is a $\mathbb{Q}$-algebraic group, and is a $\mathbb{Q}$-form of the real Lie group $U(2,1)$ of automorphisms of $\mathbb{C}^3$ preserving the Hermitian form $H$. By the construction, the Hodge group of such Hodge structures are subgroups of $U(2,1)$. And the equality holds for generic Hodge structures. The Mumford-Tate domain of a generic Hodge structure is the two-dimensional ball $\mathbb{B}$, the Hermitian symmetric domain associated with the group $U(2,1)$.
    
\end{example}

\subsubsection{Hodge representations}

Hodge representations were introduced by Green-Griffiths-Kerr \cite{GGK1} to classify the Hodge groups of polarized Hodge structures and the corresponding Mumford-Tate subdomains of a period domain.

\begin{definition}

    Let $G$ be a reductive $\Q$-algebraic group. A Hodge representation $(G,\rho,\varphi)$ is given by a $\Q$-representation $\rho:M\to \Aut(V,Q)$ and a non-constant homomorphism of Lie groups
    $\varphi:\mathbb{S}^1\to G(\mathbb{R})$
    such that $(V,Q,\rho\circ\varphi)$ is a polarized Hodge structure.
\end{definition}

Green-Griffiths-Kerr showed that the Hodge groups $G=G_{\varphi}$ and Mumford-Tate domains $D=D_{\varphi}\subset \mathcal{D}_{\textbf{h}}$ are in bijection with Hodge representations with Hodge numbers $\textbf{h}_{\varphi}\leq \mathbf{h}$. The induced (real Lie algebra) Hodge representations
$$\R\to \mathfrak{g}_{\R}\to \End(V_{\R},Q)$$
are enumerated by tuples $(\mathfrak{g}^{ss}_{\C}, E^{ss},\mu,c)$ consisting of:
\begin{itemize}[leftmargin=0.4cm]
    \item a semisimple complex Lie algebra $\mathfrak{g}_{\C}^{ss}=[\mathfrak{g}_{\C},\mathfrak{g}_{\C}]$,
    \item an element $E^{ss}\in \mathfrak{g}_{\C}^{ss}$ with the property that $\textrm{ad}\ E^{ss}$ acts on $\mathfrak{g}_{\C}^{ss}$ diagonalizably with integer eigenvalues,
    \item a highest weight $\mu$ of $\mathfrak{g}_{\C}^{ss}$, and
    \item a constant $c\in \Q$ satisfying $\mu(E^{ss})+c\in \frac{1}{2}\Z$.
\end{itemize}

Let's briefly recall the classification theorem. Given the Hodge decomposition $V_{\mathbb{C}}=\oplus V_{\varphi}^{p,q}$, the associated grading element (or infinitesimal Hodge structure) $E_{\varphi}\in i\mathfrak{g}_{\mathbb{R}}$ is defined by $E_{\varphi}(v)=\frac{1}{2}(p-q)v$ for $v\in V^{p,q}_{\varphi}$. In general, given a complex reductive Lie algebra $\mathfrak{g}_{\mathbb{C}}$, a grading element is any element $E\in\mathfrak{g}_{\mathbb{C}}$ with the property that $\mathrm{ad}(E)\in \End(\mathfrak{g}_{\C})$ acts diagonalizably on $\mathfrak{g}_{\C}$ with integer eigenvalues. That is,
$$\mathfrak{g}_{\C}=\bigoplus\limits_{\ell\in\Z}\mathfrak{g}^{\ell,-\ell}, \;\;\;\;\;\text{with}\;\;\   \mathfrak{g}^{\ell,-\ell}=\{\xi\in\mathfrak{\mathfrak{g}_{\C}}|[E,\xi]=\ell\xi\}.$$
Given the data $(\mathfrak{g}_{\C}, E)$, there is a unique real form $\mathfrak{g}_{\R}$ of $\mathfrak{g}_{\C}$ such that the above decomposition is a weight zero Hodge structure on $\mathfrak{g}_{\R}$ that is polarized by $-\kappa$. The properties of $E$ mean exactly that $(G,\mathrm{Ad}, E)$ is a Hodge representation. The data $(\mathfrak{g}_{\C},E)$ determines the homogeneous space $D_E$ and its complex dual $\check{D}_E$. More precisely, let $\mathfrak{g}=\mathfrak{z}\oplus \mathfrak{g}^{ss}$ denote the decomposition of $\mathfrak{g}$ into its center $\mathfrak{z}$ and semisimple factor $\mathfrak{g}^{ss}$. Let $E=E'+E^{ss}$ be the corresponding decomposition. The Hodge domain $D_E$ is determined by $\mathfrak{g}^{ss}$ and $E^{ss}$.


The grading element $E$ acts on any representation $G(\mathbb{R})\to \Aut(V_{\R})$ by rational eigenvalues. The $E$-eigenspace decomposition $V_{\mathbb{C}}=\oplus_{k\in\mathbb{Q}}V_k$ is a Hodge decomposition if and only if those eigenvalues lie in $\frac{1}{2}\mathbb{Z}$. More precisely, if they lie in $\Z$ (resp.~$\Z+\tfrac{1}{2}$), we can then choose any even (resp.~odd) weight and ``put'' the Hodge structure in that weight.  (The point is that $E$ only knows $p-q$, not $p+q$.)  Of course, Hodge classes have $p-q=0$ regardless of weight.

To classify the real Hodge representations, we will first need to study finite-dimensional real representations. Let $V_{\R}$ be an irreducible real representation of $G$. By Schur's lemma, $\End_G(V_{\R})$ is a division algebra over $\R$. We say that $V_{\R}$ is of real (resp. complex, quaternionic) type if $\End_G(V_{\R})$ is isomorphic to $\R$ (resp. $\C$. $\mathbb{H}$). If $V_{\R}$ is of real type, $V_{\C}$ is an irreducible self-dual representation $V_+$. If $V_{\R}$ is of complex or quaternionic type, $V_{\C}$ is the direct sum of an irreducible representation $V_+$ and its dual $V_-$, and $V_{\R}=\Res\ V_+=\Res\ V_-$. Let $\mu,\mu^*$ be the highest weight of $V_+$ and $V_-$ respectively. Then we could read off the types of the representation $V_{\R}$ from the highest weight $\mu$. See \cite{GGK1} for more details.

Furthermore, the triple $(\mathfrak{g}_{\C},E,\mu)$ is equivalent to a tuple $(\mathfrak{g}_{\C}^{ss},E^{ss},\mu^{ss},c)$ by setting $c=\mu(E')\in \Q$ and $\mu^{ss}$ the restriction of $\mu$ to $\mathfrak{g}^{ss}$.

\begin{example}
We consider the group $SU(2,1)$. Let $\lambda_i$ $(i=1,2)$ be its two fundamental weights. Let $\lambda=a\lambda_1+b\lambda_2$ $(a,b\geq0)$ be a dominant weight and $V^{\lambda}$  be the irreducible representation (over $\C$) of highest weight $\lambda$. $V^{\lambda}$ is self-dual if and only $a=b$, and $V^{\lambda}$ is the complexification of an irreducible real representation $V_{\R}^{\lambda}$ of real type. Otherwise, the dual of $V^{\lambda}$ is $V^{\lambda'}$ where $\lambda'=b\lambda_1+a\lambda_2$. The direct sum $V^{\lambda}\oplus V^{\lambda'}$ is the complexification of an irreducible real representation $\tilde{V}_{\R}^{\lambda}\cong \mathrm{Res}_{\C/\R}V^{\lambda}$ of complex type. For example, the standard representation of $SU(2,1)$ on is of complex type, with complexification the direct sum of the complex standard representation and its dual.

The case for $G=SU(3,1)$ is similar. Let $\lambda_i$ $(i=1,2,3)$ be its three fundamental weights. Let $\lambda=\sum n_i\lambda_i$ be a dominant weight and $\lambda'=n_3\lambda_1+n_2\lambda_2+n_1\lambda_3$ its dual. Then $V^{\lambda}$ is of real type if and only if $n_1=n_3$. Otherwise $V^{\lambda}\oplus V^{\lambda'}$ has a natural real structure.
\end{example}

\begin{example}
If $G=Sp_4(\R)$, all $V^{\lambda}$ are of real type. The basechange map $V_{\R}\to V_{\R}\otimes\C$ establishes a bijection from the set of irreducible real representations to the set of irreducible complex representations.
    
\end{example}

\subsubsection{VHS over locally symmetric varieties}

Let $(G,\rho,\varphi)$ be a Hodge representation. Let $D=G/K$ be the associated Mumford-Tate domain and $\mathbb{V}$ the associated VHS over $D$. Let $\Gamma$ be an arithmetic subgroup of $G$. Then $\Gamma$ acts as morphisms of Hodge structures. Therefore, the local system on the quotient $X:=\Gamma\backslash D=\Gamma\backslash G/K$ is actually a VHS. 

\begin{remark}
 Let $G'$ be a semisimple Lie group with $\mathfrak{g}'=\mathfrak{g}^{ss}$. Then $D$, as a complex manifold, is the Hermitian symmetric domain associated with $G'$. The representation $V$ is naturally a representation of $G'$, and descends to a local system over $X=\Gamma'\backslash D$ with $\Gamma'=\Gamma\cap G'$. Therefore, we may assume the group $G$ is semisimple when computing the cohomology of the local systems.
   
\end{remark}

\subsubsection{Rational Hodge twists}

Let $\mathbb{F}$ be an imaginary quadratic field $\mathbb{F}$ and $(V,\varphi)$ a real Hodge structure with $\mathbb{F}$ acting as Hodge structure morphisms. Let $\iota: \mathbb{F}\to \mathbb{C}$ be a fixed embedding. Then we define $V_+$ (resp.~$V_+^{p,q}$) to be the $\iota$-eigenspace of $V_+$ (resp. $V^{p,q}$), and $V_-$ (resp. $V_-^{p,q}$) to be the $\bar{\iota}$-eigenspace of $V_-$ (resp. $V^{p,q}$). Then we get a decomposition of two conjugate pairs of complex Hodge structures $V_+$ and $V_-$ since $\overline{V_+^{p,q}}=V_-^{q,p}$.

Now let $V=V^{\lambda}$ be an irreducible complex representation of $G$. The grading element acts on $V$ with rational eigenvalues which differ by integers.  If these eigenvalues are not half-integral, we can twist the action to make them half-integral.  More precisely, as in \cite{KK} we enlarge the MT group $G$ to $U(1)\cdot G$, with complex Lie algebra $\C\oplus \mathfrak{g}$, and replace the grading element by $\tilde{E}=(1,E)$.  We then define the twist $V\{c\}$ for $c\in \Q$ by
$$(V\{c\})_k:=V_{k+c};$$
here $c$ is simply the eigenvalue through which the ``$1$'' acts on $V$.

In the event that $\tilde{V}_{\R}$ is of complex type, with complexification $V_+\oplus V_-$, we define $\tilde{V}\{c\}_{\R}$ to be the real irrep underlying $V_+\{c\}\oplus V_-\{-c\}$.

In \cite{GGK1}, the terminiology \emph{half-twist} is used for the application of $\{-\tfrac{1}{2}\}$ to a pre-existing Hodge structure of type $V_+\oplus V_-$. This changes the parity of the weight, and is usually thought of as adding 1 to the weight.


\begin{example}\label{ex2.5}
    
    We consider the real standard representation $\tilde{V}_{\R}^{\lambda_1}$ of $U(2,1)$. Its complexification is the direct sum of the standard representation $V_+:=V^{\lambda_1}$ and its dual $V_-:=V^{\lambda_2}$. For $V_+$, $E$ acts as $2/3$ on a one-dimensional subspace, and $-1/3$ on a two-dimensional subspace. For $V_-$, $E$ acts as $1/3$ on a two-dimensional subspace, and $-2/3$ on a one-dimensional subspace.

   Taking $c=1/6$ makes the Hodge indexes half-integral as desired ($E$ acts as $\pm 1/2$). We get a level-$1$ Hodge structure on $V^{\lambda_1}\{\tfrac{1}{6}\}_{\R}$, which we can take to have weight 1. In particular, $\dim V_+^{1,0}=\dim V_-^{0,1}=1,\dim V_+^{0,1}=\dim V_-^{1,0}=2$. This is exactly the Hodge structure of a Picard curve.

    Taking a half-twist of the above Hodge structure, we get a Hodge structure of weight $2$ with Hodge numbers $(1,4,1)$. It represents the transcendental cohomology of a family of $K3$ surfaces with Picard rank $\rho=16$.
    
    Taking half-twist again, we get a Hodge structure of weight three with Hodge numbers $(1,2,2,1)$. It represents the Hodge structure of Rohde's family of Calabi-Yau threefolds.

\end{example}
\subsection{The decomposition theorem}

The aim of this paper is to prove the Hodge conjecture for certain families of varieties. Before that, we first briefly recall the perverse Leray decomposition, see \cite{CM} for details.

Let $X$ be a complex projective manifold and let $\mathcal{D}_X$ be the derived category of bounded constructible sheaves and $\mathcal{P}_X$ be the full subcategory of \textit{perverse sheaves}. To any pair $(U, L)$ where $j: U \to X$ is a Zariski open subset of $X$  and $L$ a local system over $U$, the \textit{intersection complex $IC_U(L)$} is defined as the intermediate extension $j_{!*}(L)$ and is the unique perverse extension of $L$ to $X$ with neither subobjects nor quotients supported on $X\backslash U$. Intersection complexes are simple objects in the category of perverse sheaves, and perverse sheaves are iterated extensions of intersection complexes.

The decomposition theorem studies the topological properties of proper maps between algebraic varieties.

\begin{theorem}[Decomposition theorem]
Let $f: Y\to X$ be a proper map of complex algebraic varieties. There exists an isomorphism in the constructible bounded derived category $\mathcal{D}_X$:
$$Rf_*IC_Y\cong\bigoplus\limits_{i} \ ^{\mathfrak{p}}\mathcal{H}^i(Rf_*IC_Y)[-i].$$

Furthermore, the perverse sheaves $^{\mathfrak{p}}\mathcal{H}^i(Rf_*IC_Y)$ are semisimple; i.e., there is a decomposition into finitely many disjoint locally closed and nonsingular subvarieties $X=\coprod X_{\alpha}$ and a canonical decomposition into a direct sum of intersection complexes of semisimple local systems
$$^{\mathfrak{p}}\mathcal{H}^i(Rf_*IC_Y)\cong \bigoplus\limits_{\alpha} IC_{\overline{X_\alpha}}(L_{\alpha}).$$
\end{theorem}

The \textit{intersection cohomology groups} (with middle perversity) $IH^*(X, L)$ for the local system $L$ over a Zariski open subset $U$ are simply defined to be the hypercohomology groups $H^*(X, IC_U(L))$. Taking hypercohomology of the decomposition theorem, we get:

\begin{theorem}[Decomposition theorem for intersection cohomology groups]
Let $f: Y\to X$ be a proper map of varieties. There exists finitely many triples $(X_{\alpha},L_{\alpha},d_{\alpha})$ made of locally closed, smooth and irreducible algebraic subvarieties $X_{\alpha}\subset X$, semisimple local systems $L_{\alpha}$ on $X_{\alpha}$ and integer numbers $d_{\alpha}$, such that for every open set $U\subset X$ there is an isomorphism
$$IH^r(f^{-1}U)\cong \bigoplus\limits_{\alpha} IH^{r-d_{\alpha}}(U\cap\overline{X}_{\alpha},L_{\alpha}).$$
\end{theorem}

\begin{remark}
The decomposition is not uniquely defined. But in the case when $X$ is quasi-projective, one can make distinguished choices that realize the summands as mixed Hodge substructures of a canonical mixed Hodge structure on $IH^*(Y)$. In particular, if $Y$ is smooth and $X,Y,f$ are projective, then $IH^*(Y)=H^*(Y)$ is the usual cohomology, and the intersection cohomology groups in the sum are equipped with canonical pure polarizable Hodge structures. 
\end{remark}

\subsection{$L^2$-cohomology and Zucker's conjecture}
Let $G$ be a semisimple algebraic group defined over $\mathbb{Q}$ of Hermitian type with a fixed maximal compact subgroup $K$. The associated Hermitian symmetric domain $D$ is $G(\mathbb{R})/K$. Let $\Gamma$ be an arithmetic subgroup (or a congruence subgroup) of $G$ and let $X=\Gamma\backslash D=\Gamma\backslash G(\mathbb{R})/K$ be a locally symmetric variety. Then $X$ is a quasi-projective variety with Baily-Borel compactification $\overline{X}$. Denote by $i: X\to\overline{X}$ the natural inclusion map. Let $(V,\rho)$ be a (rational) representation of $G$, it defines a local system $\mathbb{V}$ over $X$. We are interested in the intersection cohomology $IH^*(\overline{X},\mathbb{V})$.

The Hermitian symmetric domain $D$ is equipped with a canonical Riemannian metric induced from the Killing form of the Lie algebra $\mathfrak{g}$. This metric is $\Gamma$-invariant, thus descends to a Riemannian metric over $X=\Gamma\backslash D$. We also choose and fix a metric on the local system $\mathbb{V}$. The $L^2$-cohomology groups $H^*_{(2)}(X,\mathbb{V})$ are defined to be the cohomology groups of the complex $(C^{\bullet},d)$, where $C^k$ is the space of $\mathbb{V}$-valued smooth $k$-forms over $X$ such that the form itself and its exterior derivative are both square-integrable; the differential map $d$ is simply the restriction of the usual exterior differential. Zucker's conjecture compares the intersection cohomology over $\overline{X}$ and $L^2$-cohomology over $X$. It was proved in different ways by Eduard Looijenga (\cite{Loo}) and by Leslie Saper and Mark Stern (\cite{SS}).

\begin{theorem}[Zucker's conjecture]
As real vector spaces, the intersection cohomology $IH^*(\overline{X},\mathbb{V}_{\R})$
is isomorphic to the $L^2$-cohomology $H^*_{(2)}(X,\mathbb{V}_{\R})$.
\end{theorem}

Let $G$ be a Lie groups and $U$ be a $(\mathfrak{g},K)$-module. The relative Lie algebra cohomology groups are the cohomology groups of the complex $(C^{\cdot},d)$ where $C^q=\Hom_\mathfrak{k}(\wedge^q(\mathfrak{g}/\mathfrak{k}), U)$. See \cite{BW} for more details and properties. It is well-known (de Rham isomorphism) that we may compute the cohomology groups of local systems by relative Lie algebra cohomology:
$$H^*(X,\mathbb{V})\cong H^*(\mathfrak{g},K,C^{\infty}(\Gamma\backslash G(\mathbb{R}))\otimes V).$$
The $L^2$-cohomology groups have a similar description
$$H^*_{(2)}(X,\mathbb{V})\cong H^*(\mathfrak{g},K,L^2(\Gamma\backslash G(\mathbb{R}))^{\infty}\otimes V).$$
The inclusion  $L^2(\Gamma\backslash G(\mathbb{R}))^{\infty}\to C^{\infty}(\Gamma\backslash G(\mathbb{R}))$ induces a map $$H^*(\mathfrak{g},K,L^{2}(\Gamma\backslash G(\mathbb{R}))^{\infty}\otimes V)\to H^*(\mathfrak{g},K,C^{\infty}(\Gamma\backslash G(\mathbb{R}))\otimes V),$$ which recovers the natural map
$H^*_{(2)}(X,\mathbb{V})\to H^*(X,\mathbb{V})$
under the above isomorphisms.

\vspace{.3in}

The unitary representation $L^2(\Gamma\backslash G)$ is the direct sum of a discrete spectrum and a continuous spectrum. The discrete spectrum is a Hilbert direct sum of unitary representations $U_{\pi}$, each with finite multiplicity $m_{\pi}(\Gamma)$. The contribution from the continuous spectrum to the $(\mathfrak{g}, K)$-cohomology vanishes if $G$ has a discrete series. This is always true in our paper. Let $\mathcal{A}(G,\Gamma)$ be the space of automorphic forms with respect to $\Gamma$ and $\mathcal{A}^2(G,\Gamma)=\mathcal{A}(G,\Gamma)\cap L^2(\Gamma\backslash G)$. Then the decomposition of $L^2_{\mathrm{disc}}(\Gamma\backslash G)$ is the same as the decomposition of $\mathcal{A}^2(G;\Gamma)$ as $\mathcal{A}^2(G;\Gamma)$ is precisely the space of smooth $K$-finite, $Z(\mathfrak{g})$-finite vectors in the discrete spectrum $L^2_{\mathrm{disc}}(\Gamma\backslash G)$. In summary, 
 the $L^2$-cohomology groups remain unchanged if we replace $L^2(\Gamma\backslash G(\mathbb{R}))^{\infty}$ with $\mathcal{A}^2(G;\Gamma)$, the subspace of $L^2$-automorphic forms:

$$H^*_{(2)}(X,\mathbb{V})=H^*(\mathfrak{g},K,\mathcal{A}^2(G;\Gamma)\otimes V).$$

Combined with Zucker's conjecture, the intersection cohomology  can be interpreted as relative Lie algebra cohomology:
$$IH^*(\overline{X},\mathbb{V})=H^*(\mathfrak{g},K,\mathcal{A}^2(G;\Gamma)\otimes V)=\bigoplus\limits_{U_{\pi}\in \mathcal{A}^2(G;\Gamma)} m_{\pi}(\Gamma)H^*(\mathfrak{g},K,U_{\pi}\otimes V).$$

The discrete spectrum is the direct sum of the cuspidal and the residual spectrum. Correspondingly, the intersection cohomology is the direct sum of a cuspidal part $IH^*_{\mathrm{cusp}}(\overline{X},\mathbb{V})$ and a residual part $IH^*_{\mathrm{res}}(\overline{X},\mathbb{V})$. The natural map $IH^*(\overline{X,},\mathbb{V})\to H^*(X,\mathbb{V})$ is injective on the cuspidal part, but the restriction to the residual part is usually neither injective nor surjective.

\vspace{1cm}

Zucker's conjecture is just an isomorphism of $\R$-vector spaces. However, both sides carry natural $\R$-Hodge structures. The Hodge structure of the intersection cohomology comes from Saito's mixed Hodge module theory (and is defined over $\QQ$) while the Hodge structure of $L^2$-cohomology comes from harmonic analysis. It is an open question whether the isomorphism is actually an isomorphism of Hodge structures.

\begin{hzconj}
The isomorphism in Theorem 2.8 is an isomorphism of $\mathbb{R}$-Hodge structures.
\end{hzconj}

We know H-Z in several cases:
\begin{itemize}[leftmargin=0.4cm]
\item $\Gamma$ cocompact: because $Y=Y^*$ (no boundary);
\item Hilbert modular surface or complex $n$-balls (Zucker \cite{Z}), trivial coefficient $\mathbb{C}$;
\item Hilbert modular varieties \cite{MSYZ};
\item in general, for $i\leq c:=$ codimension of singular locus in $Y^*$.  In this case actually $\IH^i(Y^*,\mathcal{H}^{2p-i})=W_{2p}H^i(Y,\mathcal{H}^{2p-i})=H^i_{(2)}(Y^*,\mathcal{H}^{2p-i})$ \cite[\S5]{HZ}.
\item another general fact is that the HSs on $H^*_{(2)}$ and $\IH^*$ are equal if we replace $Y^*$ by $\hat{Y}$, though this isn't really a case of H-Z.  It is due to Cattani-Kaplan-Schmid, Kashiwara-Kawai, and Saito (cf.~\cite{PSa}).
\end{itemize}
Even in the absence of the H-Z conjecture, the Hodge number computations that follow in \S\ref{S3} still provide upper bounds on the Hodge numbers of $W_{2p}H^i(Y,\mathcal{H}^{2p-i})$.  So they do mean something geometric.  In the case $i\leq c$, they exactly compute the Hodge numbers of this space.

\subsection{The Hodge conjecture}\label{S2.4}
Write $\Gamma(H):=\mathrm{Hom}_{\text{MHS}}(\mathbb{Q}(0),H)$ for Hodge $(0,0)$-classes in a MHS $H$.

\begin{conjecture}
Let $Y/\mathbb{C}$ be a quasi-projective algebraic variety (possibly neither smooth nor complete) of dimension $d$.  Then the cycle class map
$$\mathrm{CH}^p(Y)\to \Gamma (H^{\textup{BM}}_{2(d-p)}(Y,\mathbb{Q}(p-d)))\;\left(\cong\; \Gamma(H^{2p}(Y,\mathbb{Q}(p)))\text{ for $Y$ smooth}\right)$$
is surjective.\footnote{Here $H_{k}^{\text{BM}}(Y)$ denotes Borel-Moore homology \cite{PSt}, which carries a natural MHS.  It is equal to the relative homology group $H_{k}(\bar{X},\bar{X}\setminus X)$ for any compactification $\bar{X}$, but is independent of the choice of $\bar{X}$.} We abbreviate this conjecture by $\textup{HC}^p(X)$.
\end{conjecture}

Now fix a smooth quasi-projective $X$, a possibly singular compactification $\bar{X}$ and a resolution $\hat{X}\overset{\rho}{\twoheadrightarrow}\bar{X}$.  Write $Z=\bar{X}\setminus X$ and $\hat{Z}=\hat{X}\setminus X$ for the complements; assume that $\hat{Z}$ is a normal-crossing divisor.  Statements about the relationship between $X$ and $\bar{X}$ will obviously apply to $\hat{X}$ too, since $\bar{X}$ is more general.

\begin{proposition}\label{p1}
\textbf{(i)} $\HC^p(\hat{X}) \implies \HC^p(\bar{X}).$

\noindent\textbf{(ii)} $\HC^p(\bar{X})\implies \HC^p(X).$

\noindent\textbf{(iii)} $\HC^p(X)+\HC^p(Z)\implies \HC^p(\bar{X}).$
\end{proposition}

\begin{proof}[Sketch]
Writing $q=d-p$, (i) is essentially because $$W_{0}H_{2q}^{\text{BM}}(\bar{X},\Q(q))=\mathrm{gr}^W_0H^{2p}(\bar{X},\mathbb{Q}(p))\hookrightarrow H^{2p}(\hat{X},\Q(p)),$$ which gives injectivity on Hodge classes, whereas cycles push down under $\rho$.  For (ii) and (iii), the main idea (as in \cite{Ar2}) is the diagram
$$\xymatrix{\mathrm{Hg}^p(Z)\ar [r] & \mathrm{Hg}^p(\bar{X}) \ar [r] & \mathrm{Hg}^p(X)\ar [r] & 0\\ \mathrm{CH}^p(Z) \ar [r] \ar [u] & \mathrm{CH}^p(\bar{X}) \ar [r] \ar [u] & \mathrm{CH}^p(X) \ar [r] \ar [u] & 0.}$$
with exact rows.
\end{proof}

The conclusion here is that one can focus on the HC for $X$ without further discussion, because (a) it has a well-defined meaning and (b) the boundary in a compactification should be thought of as a separate issue (or a non-issue if $\dim(X)\leq 4$).

Now consider $X\overset{\pi}{\twoheadrightarrow} Y$ a smooth projective morphism, with $X$ and $Y$ smooth quasi-projective.  We also have compactifications $\bar{\pi}\colon \bar{X}\to \bar{Y}$, $\hat{\pi}\colon \hat{X}\to \hat{Y}$ to projective morphisms, with assumptions as above ($\bar{Y},\bar{X}$ possibly singular; $\hat{Y},\hat{X}$ smooth with NC boundary).  Writing $\mathcal{H}^r:=R^r\pi_*\mathbb{Q}_X$, there are cycle maps from $\mathrm{CH}^p(X)$ to $\Gamma(H^i(Y,\mathcal{H}^{2p-i}(p)))$ by splitting the cycle-class into Leray graded pieces.  Let $\HC^p_i(\pi)$ stand for surjectivity of these maps.  

Moreover, by the Decomposition Theorem for $\hat{\pi}$ (e.g.~see \cite{KL} for a convenient presentation), we get cycle maps $\mathrm{CH}^p(\hat{X})\to \Gamma(\IH^i(\hat{Y},\mathcal{H}^{2p-i}(p))).$  We can also apply the DT to the composition $\bar{\pi}\circ \rho$ to get maps $\mathrm{CH}^p(\hat{X})\to \Gamma(\IH^i(\bar{Y},\mathcal{H}^{2p-i}(p)))$. Write $\HC^p_i(\hat{\pi})$, $\HC^p_i(\bar{\pi})$ for surjectivity of these maps.

\begin{proposition}\label{p2}
\textbf{(i)} $\HC^p_i(\pi)\;(\forall i)\implies \HC^p(X).$

\noindent \textbf{(ii)} $\HC^p_i(\pi)\;(\forall (p,i)\in Q)\implies \HC^p(X)\;(\forall p)$, where $Q$ is the parallelogram defined by $0\leq i\leq d_Y$ and $\tfrac{i}{2}\leq p\leq \tfrac{i+(d_X-d_Y)}{2}$. 
\end{proposition}

\begin{proof}[Discussion]
See \cite{Ar2} for the proof.  One key point is that $\IH^i(\hat{Y},\mathcal{H}^{2p-i})\twoheadrightarrow W_{2p}H^i(Y,\mathcal{H}^{2p-i})$, which means that Hodge classes surject and $\HC^p_i(\hat{\pi})\implies \HC^p_i(\pi)$.  (But this works for $\bar{\pi}$ too as we're about to see.)  More precisely:  a Hodge class in $W_{2p}H^i(Y,\mathcal{H}^{2p-i})$ lifts (nonuniquely) to a Hodge class in $\mathrm{IH}^i(\hat{Y},\mathcal{H}^{2p-i})$ (or $\bar{Y}$); if that is given by a cycle on $\hat{X}$, then obviously the original class is given by its restriction to $X$. 
\end{proof}

\begin{proposition}\label{p3}
\textbf{(i)} $\HC^p_i(\hat{\pi})\implies \HC^p_i(\bar{\pi})$.

\noindent \textbf{(ii)}	$\HC^p_i(\bar{\pi})\implies \HC^p_i(\pi).$

\noindent \textbf{(iii)} $\HC^p_i(\bar{\pi})\;(\forall (p,i)\in Q)\implies \HC^p(X)\;(\forall p).$
\end{proposition}

\begin{proof}
By the DT, $\IH^i(\bar{Y},\mathcal{H})\subseteq \IH^i(\hat{Y},\mathcal{H})$ is a sub-HS; so classes lift and cycles push forward under $\rho$.  This gives (i).

(ii) is the most subtle.  Clearly, it is enough to show
\begin{equation}\label{e1}
\IH^i(\bar{Y},\mathcal{H}^{2p-i})\twoheadrightarrow W_{2p}H^i(Y,\mathcal{H}^{2p-i})
\end{equation}
so that Hodge classes lift. Essentially the same thing is proved in \cite{PSa} (take the algebraic case of their main theorem).  Here is a brief recap.

If we write $\mathsf{H}^{2p-i}$ for the IC-Hodge-module on $\bar{Y}$ associated to $\mathcal{H}^{2p-i}$, shifted so that the restriction to $Y$ is $\mathcal{H}^{2p-i}$ (not $\mathcal{H}^{2p-i}[d_Y]$), then we have the localization sequence
$$\to \mathbb{H}^i(\bar{Y},\mathsf{H}^{2p-i})\to \mathbb{H}^{i}(\bar{Y},R\jmath_*\jmath^*\mathsf{H}^{2p-i})\to \mathcal{H}^{i+1}(\bar{Y},\imath_*\imath^!\mathsf{H}^{2p-i})\to$$
whose first 2 terms are $\IH^i(\bar{Y},\mathcal{H}^{2p-i})$ and $H^i(Y,\mathcal{H}^{2p-i})$. Since $\imath^!$ and $\imath_*$ do not decrease weights (cf.~\cite[\S14.1.1]{PSt}), $W_{2p}$ of the right-hand term is zero.  Taking $W_{2p}$ of the sequence then gives the desired result.

(iii) is just putting (ii) together with Proposition \ref{p2}(ii).
\end{proof}

Consider the case where $Y$ is a locally symmetric variety, $\bar{Y}=Y^*$ is the BB compactification, and (say) $\hat{Y}$ is a toroidal one; $\hat{X}$ is a smooth compactification of a family $X$ giving one of the Hermitian VHSs over $Y$.  (I'll continue to write $\bar{\pi}$ rather than $\pi^*$ which looks like a pullback.)  The objects you are computing in your paper are the $L^2$-cohomology groups $H^i_{(2)}(Y^*,\mathcal{H}^{2p-i})$, which carry $\mathbb{R}$-HSs coming from (equivalently) $(\mathfrak{g},\mathfrak{k})$-cohomology or harmonic forms.  Everything stated so far is about the IH-groups because they are the ones the DT relates to cohomologies of the total space hence the HC.

\begin{proposition}\label{p4}
\textbf{(i)} $H^i_{(2)}({Y}^*,\mathcal{H}^{2p-i})\cong \IH^i({Y}^*,\mathcal{H}^{2p-i})$ as real vector spaces, and both have $\mathbb{R}$-MHS morphisms to $H^i({Y},\mathcal{H}^{2p-i})$ with the same image, namely $W_{2p}H^i({Y},\mathcal{H}^{2p-i})$.

\noindent \textbf{(ii)} If $(H^i_{(2)}(Y^*,\mathcal{H}^{2p-i}))^{p,p}=\{0\}$, then $\HC^p_i(\pi)$ holds.

\noindent \textbf{(iii)} If $(H^i_{(2)}(Y^*,\mathcal{H}^{2p-i}))^{p,p}=\{0\}$ for all $(p,i)\in Q$, then $\HC^p(X)$ holds.
\end{proposition}

\begin{proof}[Sketch]
(i) is Theorem 5.4 of \cite[\S5]{HZ}.  It says we can write (as $\mathbb{R}$-HSs)
\begin{align*}
H^i_{(2)}(Y^*,\mathcal{H}^{2p-i})&\cong W_{2p}H^i(Y,\mathcal{H}^{2p-i})\oplus M' \\
\IH^i(Y^*,\mathcal{H}^{2p-i})&\cong W_{2p}H^i(Y,\mathcal{H}^{2p-i})\oplus M''
\end{align*}
for some weight-$2p$ $\mathbb{R}$-Hodge structures $M'$ and $M''$.

To show (ii), note that the hypothesis implies $\Gamma(H^i(Y,\mathcal{H}^{2p-i}(p)))=\{0\}$.  So $\HC^p_i(\pi)$ is true \emph{vacuously}.

Clearly (iii) is a corollary of (ii) together with Proposition 2(ii).
\end{proof}

So (ii)-(iii) make precise \emph{what} one is proving when checking there are no $\mathbb{R}$-Hodge classes in the relative Lie algebra cohomologies below:  namely, the HC for $X$.  This conclusion does \emph{not} depend on knowledge of the Harris-Zucker conjecture.

\section{Computation of Hodge numbers}\label{S3}

\subsection{Vogan-Zuckerman classification}

Let $G$ be a  semisimple Lie group. Fix a maximal compact subgroup $K$. This is equivalent to choosing a Cartan involution $\theta$ or the Cartan decomposition $\mathfrak{g}=\mathfrak{k}\oplus \mathfrak{p}$. Here we use $\mathfrak{g}_0$ (resp. $\mathfrak{k}_0$) to represent the \textit{real} Lie algebra of $G$ (resp. $K$), and $\mathfrak{g}$ (resp. $\mathfrak{k}$) to represent their complexifications. We always assume that $G$ has a maximal compact torus $T$, which is equivalent to that $G$ has discrete series representations. Let $\mathfrak{t}_0$ be the Lie algebra of $T$, it is a Cartan subalgebra of both $K$ and $G$.

Let $(\pi,U)$ be a $G$-representation. We will not distinguish the representation $U$  and its associated $(\fg,K)$-module. The \textit{relative Lie algebra cohomology groups} $H^*(\mathfrak{g},K,U)$ are defined to be the cohomology groups of the complex $(C^*(\fg, K, U),d)$ where $C^*(\fg,K,U)=\Hom_{K}(\wedge^{*}(\fg/\mathfrak{k}), U)$ (\cite{BW}).

Let $(\pi,U)$ be an irreducible unitary $(\mathfrak{g},K)$-module and $V$ be a finite-dimensional irreducible representation of $G$. We say that $(\pi,U)$ is \textit{cohomological (w.r.t $V$)} if the $(\mathfrak{g},K)$-cohomology groups $H^{*}(\mathfrak{g},K;U\otimes V)\neq 0$. First, we have:
\begin{lemma}
If $H^*(\fg, K, U\otimes V)\neq 0$, then the infinitesimal character of $U$ is the same as the infinitesimal character of $V^*$, and the differential map $d$ in the complex $(C^*(\fg,K, U\otimes V),d)$ is automatically zero.
\end{lemma}

Cohomological $(\mathfrak{g}, K)$-modules are classified by Vogan and Zuckerman (\cite{VZ}) in terms of $\theta$-stable parabolic subalgebras. A $\theta$-stable parabolic subalgebra $\mathfrak{q}=\mathfrak{q}(X)\subset \mathfrak{g}$ is associated to an element $X\in i\mathfrak{t}_0$. It is defined as the direct sum 
$$\mathfrak{q}=\mathfrak{l}\oplus \mathfrak{u},$$
of the centralizer $\mathfrak{l}$ of $X$ and the sum $\mathfrak{u}$ of the positive eigenspaces of $\mathrm{ad}(X)$ ($\mathfrak{t}\subset \mathfrak{l}$ by construction). Since $\theta(X)=X$, the subspace $\mathfrak{q},\mathfrak{l}$ and $\mathfrak{u}$ are all invariant under $\theta$, so is
$$\mathfrak{q}=\mathfrak{q}\cap\mathfrak{k}\oplus \mathfrak{q}\cap\mathfrak{p},$$
and so on.

The Lie algebra $\mathfrak{l}$ is the complexification of $\mathfrak{l}_0=\mathfrak{l}\cap \mathfrak{g}_0$. Let $L$ be the connected subgroup of $G$ with Lie algebra $\mathfrak{l}_0$. Fix a positive system $\Delta^+(\mathfrak{l})$ of the roots of $\mathfrak{t}$ in $\mathfrak{l}$. Then $\Delta^+(\mathfrak{g},\mathfrak{t})=\Delta^+(\mathfrak{l})\cup \Delta(\mathfrak{u})$ is a positive system of the roots of $\mathfrak{t}$ in $\mathfrak{g}$. Let $\rho$ be half the sum of the roots in $\Delta^+(\mathfrak{g},\mathfrak{t})$ and $\rho(\mathfrak{u}\cap\mathfrak{p})$ half the sum of roots in $\mathfrak{u}\cap\mathfrak{p}$. 

A one-dimensional representation $\lambda:\mathfrak{l}\to\mathbb{C}$ is called \textit{admissible} if it satisfies the following conditions:

\begin{enumerate}[leftmargin=0.8cm]
\item[(1)]$\lambda$ is the differential of a unitary character of $L$,
\item[(2)]if $\alpha \in\Delta(\mathfrak{u})$, then $\langle\alpha,\lambda|_{\mathfrak{t}}\rangle\geq 0$.
\end{enumerate}

Given $\mathfrak{q}$ and an admissible $\lambda$, let $\mu(\mathfrak{q},\lambda)$ be the representation of $K$ of highest weight $\lambda|_{\mathfrak{t}}+2\rho(\mathfrak{u}\cap\mathfrak{p})\footnote{The restriction of $\Delta^+(\mathfrak{g},\mathfrak{t})$ to $\mathfrak{k}$ is a positive system. This weight is dominant with respect to this positive system.}$.

\begin{proposition}
There exists a unique irreducible unitary $(\mathfrak{g},K)$-module $A_{\mathfrak{q}}(\lambda)$ such that
\begin{enumerate}[leftmargin=0.8cm]
\item[\textup{(1)}]$A_{\mathfrak{q}}(\lambda)$ contains the $K$-type $\mu(\mathfrak{q,\lambda})$.
\item[\textup{(2)}]$A_{\mathfrak{q}}(\lambda)$ has infinitesimal character $\lambda|_{\mathfrak{t}}+\rho$.
\end{enumerate}
\end{proposition}


\begin{remark}
    The $K$-representation $\mu(\mathfrak{q},\lambda)$ is minimal in the sense that all the $K$-types of $A_{\mathfrak{q}}(\lambda)$ are of the form
    $$\delta=\lambda|_{\mathfrak{t}}+2\rho(\mathfrak{u}\cap\mathfrak{p})+\sum\limits_{\beta\in \Delta(\mathfrak{u}\cap\mathfrak{p})}n_{\beta}\beta$$
    with $n_{\beta}$ non-negative integers.
\end{remark}

\begin{remark}
    
    We use the simplified symbols $\mu(\mathfrak{q})$, $A_{\mathfrak{q}}$ to denote $\mu(\mathfrak{q},0)$ and $A_{\mathfrak{q}}(0)$. It turns out that $\mu(\mathfrak{q})$ are actually $K$-representations inside the natural representation $\wedge^*\mathfrak{p}$ (\cite{VZ}).
\end{remark}

\begin{remark}
 Given a finite-dimensional representation $V$, the Zuckerman translation functor is a functor twisting the infinitesimal character of represntations. The representations $A_{\mathfrak{q}}(\lambda)$ are exactly the Zuckerman translations of the representations $A_{\mathfrak{q}}$.
\end{remark}

\begin{proposition}Let $(\pi,U)$ be an irreducible unitary $(\mathfrak{g},K)$-module and $V$ be a finite dimensional irreducible representation of $G$. Suppose $H^*(\mathfrak{g},K;U\otimes V)\neq 0$. Then there is a $\theta$-stable parabolic subalgebra $\mathfrak{q}=\mathfrak{l}\oplus\mathfrak{u}$ of $\mathfrak{g}$, such that
\begin{enumerate}[leftmargin=0.8cm]
\item[\textup{(1)}] $V/\mathfrak{u}V$ is a one-dimensional unitary representation of $L$; write $-\lambda:\mathfrak{l}\to \mathbb{C}$ for its differential.
\item[\textup{(2)}]$\pi\cong A_{\mathfrak{q}}(\lambda)$. Moreover, letting $R=\dim(\mathfrak{u}\cap\mathfrak{p})$, we have
$$H^*(\mathfrak{g},K;U\otimes V)=H^{*-R}(\mathfrak{l},L\cap K;\mathbb{C})=\Hom_{\mathfrak{l}\cap\mathfrak{k}}(\wedge^{*-R}(\mathfrak{l}\cap\mathfrak{p}),\mathbb{C}).$$
\end{enumerate}
\end{proposition}

\subsubsection{A Hodge number formula}

When $G/K$ is Hermitian symmetric, the cohomology of a representation acquires a bigrading as follows. Write $\mathfrak{p}^+\subset \mathfrak{p}$ (resp. $\mathfrak{p}^-\subset \mathfrak{p}$) for the holomorphic (resp. anti-holomorphic) tangent space of $G/K$ at the origin. Then the decomposition $$\wedge^*\mathfrak{p}=\oplus (\wedge^*\mathfrak{p}^+)\otimes (\wedge^*\mathfrak{p}^-)$$
induces a bigrading on the complex $\Hom_{\mathfrak{p}}(\wedge^*\mathfrak{p},U\otimes V)$. For a cohomological representation, the differential map $d$ is zero, so this bigrading descends to a bigrading on the cohomological  spaces:
$$H^i(\mathfrak{g},K;U\otimes V)=\bigoplus\limits_{p+q=i}H^i(\mathfrak{g},K;U\otimes V)^{p,q}.$$
This decomposition corresponds to the Hodge structure on the cohomology of locally symmetric spaces, and the computation is given in \cite{VZ}. (That is, we are treating $V$ as a trivial HS of type $(0,0)$; this will be fixed in a moment.) The Hodge types are calculated in terms of the decomposition of the $\theta$-stable parabolic subalgebra $\mathfrak{q}$. We denote the Hodge type given in \cite{VZ} by $(p_U,q_U)$.


Now if $V$ is a representation of real type of the group $G$, and the associated local system $\V$ is a VHS of weight $w$, then we need to shift the Hodge numbers to make $H^i(\mathfrak{g},K;U\otimes V)$ a HS of weight $i+w$, viz.
$$H^i(\mathfrak{g},K;U\otimes V)=\bigoplus\limits_{p+q=i+w}H^i(\mathfrak{g},K;U\otimes V)^{p,q}.$$
A class on the left-hand side can be written as $\sum_j \omega_j\otimes u_j\otimes v_j$ where $\omega_j\in (\wedge^*\mathfrak{p})^{\vee}$ is a ``differential form'' of type $(p_U,q_U)$, while $v_j\in V$ and $u_j\in U$. From the construction of the representations $A_{\mathfrak{q}}$ and the computation of their cohomology, all the vectors $v_j$ are in the same $K$-sub-representation of $V$, hence have the same Hodge type $(p'_U,q'_U)$.  So the Hodge type $(p,q)$ for such a class is given by $(p_U+p'_U,q_U+q'_U)$. 

Taking the sum of all possible representations, we get the following Hodge number formula for local systems.

\begin{proposition}\label{p3.4}
  Set $h_i^{p,q}(\pi,V):=\dim(H_{(2)}^i(X^*,\mathbb{V})^{p,q})$, then 
   $$h_i^{p,q}(\pi,V)=\sum \textrm{mult.}(\Gamma,U_{\pi})\dim(H^i(\mathfrak{g},K;U_{\pi}\otimes V))$$
   where $U_{\pi}$ sums over cohomological representations with nonzero $H^i$ and such that $(p_U+p'_U,q_U+q'_U)=(p,q)$.
\end{proposition}

Now if $V$ is a representation of complex or quaternionic type of the group $G$, and the associated local system $\V$ is a VHS.  Then $V_{\C}=V_{\mu}\oplus V_{\mu^*}$, and the computation of the Hodge types is the same as in the real case, but it is the direct sum of two representations. 
\begin{proposition}
Set $h_i^{p,q}(\pi,V):=\dim(H_{(2)}^i(X^*,\mathbb{V})^{p,q})$; then 
\begin{align*}
h_i^{p,q}(\pi,V)=&\sum \textrm{mult.}(\Gamma,U_{\pi})\dim(H^i(\mathfrak{g},K;U_{\pi}\otimes V_{\mu}))\\ &+\sum \textrm{mult.}(\Gamma,U^*_{\pi})\dim(H^i(\mathfrak{g},K;U^*_{\pi}\otimes V_{\mu^*}))
\end{align*}
   where $U_{\pi}$ (resp. $U^*_{\pi}$) sums over cohomological representations with respect to $V$ (resp. $V^*$) with nonzero $H^i$ and such that $(p_U+p'_U,q_U+q'_U)=(p,q)$ (resp. $(p_{U^*}+p'_{U^*},q_{U^*}+q'_{U^*})=(p,q)$). 
\end{proposition}

If $\V$ is not a VHS, we could still make a rational half-twist to get a VHS. Let $c$ be the shifting constant. We could pretend that the vector has a rational Hodge type $(p'_U,q'_U)$, and then add the twisting constant $c$.

\begin{proposition}
Set $h_i^{p,q}(\pi,V):=\dim(H_{(2)}^i(X^*,\mathbb{V})^{p,q})$; then 
\begin{align*}
h_i^{p,q}(\pi,V)=&\sum \textrm{mult.}(\Gamma,U_{\pi})\dim(H^i(\mathfrak{g},K;U_{\pi}\otimes V_{\mu}))\\ &+\sum \textrm{mult.}(\Gamma,U^*_{\pi})\dim(H^i(\mathfrak{g},K;U^*_{\pi}\otimes V_{\mu^*}))
\end{align*}
where $U_{\pi}$ (resp. $U^*_{\pi}$) sums over cohomological representations with respect to $V$ (resp. $V^*$) with nonzero $H^i$ and such that $(p_U+p'_U+c,q_U+q'_U+c)=(p,q)$ (resp. $(p_{U^*}+p'_{U^*}-c,q_{U^*}+q'_{U^*}-c)=(p,q)$). 
\end{proposition}

\begin{remark}
In practice, the group $G$ could be reductive. Since the cohomology of local systems is not changed if we study the semisimple subgroup. We could calculate everything in the semisimple setting. Just note that, to calculate the Hodge type of the vector, $v$ should be considered as a representation of $G$. 
\end{remark}

\begin{remark}
It seems that it is better to study packets of representations if representations in the same packet occurs with the same multiplicity, and have `conjugate' Hodge numbers. 
\end{remark}

\subsection{Some computations}

\subsubsection{$SU(2,1)$}

As a Lie group, $SU(n,1)$ is the group of matrices $M$ preserving the sesquilinear form defined by $\mathrm{diag}(I_n,-1)$: $M\mathrm{diag}(I_n,-1)\overline{M}^t=\mathrm{diag}(I_n,-1)$. The maximal torus of $SU(n,1)$ is the two-dimensional torus consisting of diagonal matrices \{$\mathrm{diag}(e^{2\pi i\theta_1}, e^{2\pi i\theta_2},\cdots, e^{2\pi i\theta_n}),\sum \theta_i\equiv 0\mod 2\pi$\}. The maximal compact subgroup is isomorphic to $U(n)$, consisting of matrices $\mathrm{diag}(A,\det(A)^{-1})$ where $A$ is in $U(n)$.

The natural $K$-representation $\wedge^*\mathfrak{p}$ has a simple description \cite{BW}. Let $\tau$ be the standard representation of $U(n)$ on $\C^n$. Set $\tau_1=(\det)\otimes \tau$. Then as a $K$-representation, $\mathfrak{p}\cong \tau_1\oplus \tau_1^*$. Wedge powers of $\tau_1$ and $\tau^*_1$ are irreducible. And 
$$\wedge^q\mathfrak{p}=\bigoplus\limits_{r+s=q}\Lambda^{r,s}$$ with $K$ acting on $\Lambda^{r,s}$ by $\wedge^r\tau_1\otimes\wedge^s\tau_1^*$. Note that $\Lambda^{r,s}$ are not necessarily irreducible.

Irreducible representations of $U(2)$ are finite-dimensional. We have a surjective product map from $SU(2)\times U(1)$ to $U(2)$ with kernel $\{\pm 1\}$. The tensor product of $k$-th symmetric power ($k\in\mathbb{N}$) of the standard representation of $SU(2)$ and $l$-th character ($l\in\mathbb{Z}$) of $U(1)$ descends to an irreducible representation of $U(2)$ if and only if $k\equiv l\mod 2$. Therefore, the irreducible representations of $U(2)$ are parameterized by the pairs $(k,l)$ with $k\in\mathbb{N},l\in\mathbb{Z}$ and $k\equiv l\mod 2$. The representation $\tau_1$ is $(1,1)$, and $\tau_1^*$ is $(1,-1)$.

The complexified Lie algebra $\mathfrak{su}(2,1)$ is isomorphic to $\mathfrak{sl}(3)$. Following notations in \cite{GGK2}, its simple roots are $\alpha_1=e_2^*-e_1^*$ and $\alpha_2=e_3^*-e_2^*$. The root $\alpha_1$ is a compact root. From Weyl's unitary trick, finite-dimensional representations of $SU(2,1)$ are the same as finite-dimensional representations of $\mathfrak{su}(2,1)\cong\mathfrak{sl}(3)$. There are two fundamental representations corresponding to the two dominant weights: $\lambda_1=\frac{1}{3}(\alpha_1+2\alpha_2)$ corresponds to the standard representation $W$, and $\lambda_2=\frac{1}{3}(2\alpha_1+\alpha_2)$ corresponds to the dual $W^*\cong \wedge^2W$. Other representations are subrepresentations of tensor powers of $W$ and $W^*$. 

\vspace{0.5cm}

Now we list cohomological representations. Let $V$ be a finite-dimensional representation of $SU(2,1)$.

We first consider the case when $V$ is the trivial representation. We have three types of cohomological representations (\cite{BW}):
\begin{itemize}[leftmargin=0.4cm]
\item The three discrete series representations with the same infinitesimal characters as the trivial representation. One is the holomorphic discrete series $D^+$, with minimal $K$-type $(0,2)$; one is the anti-holomorphic discrete series $D^-$, with minimal $K$-type $(0,-2)$; and one is the non-holomorphic discrete series $D^0$; with minimal $K$-type $(2,0)$. The only nonvanishing cohomology is $H^2(\mathfrak{g},K,U)=\mathbb{C}$. 
\item The two non-tempered representations $J_{1,0}$ and $J_{0,1}$. The minimal $K$-type of $J_{1,0}$ is $(1,1)$, and the minimal $K$-type of $J_{0,1}$ is $(1,-1)$. The nontrivial cohomology is $H^1(\mathfrak{g},K,U)=H^3(\mathfrak{g},K,U)=\mathbb{C}$. 
\item The trivial representation. The minimal $K$-type is $(0,0)$. The cohomology groups are those coming from the base varieties.
\end{itemize}

If $V$ is regular, the only cohomological representations are discrete series representations. 

If $V=V^{\lambda}$ is singular, then $\lambda=n\lambda_1$ or $n\lambda_2$. In this case, besides discrete series we have some nontempered cohomological representations. If $\lambda=n\lambda_1$ (resp. $n\lambda_2$), a twist of $J_{1,0}$ (resp. $J_{0,1}$) is also cohomological. We call it $J^n_{1,0}$ (resp. $J^n_{0,1}$).  In order to list the corresponding Hodge types for $V=V^{n\lambda_1}\subseteq (V^{\lambda_1})^{\otimes n}$, we have to fix the Hodge structure on $V^{\lambda_1}$ as in Example \ref{ex2.5}.  First, we take the Hodge type of $V^{\lambda_1}$ to be twisted so that $h^{1,0}=1$ and $h^{0,1}=2$, so that $\mathcal{H}^1$ of the universal Picard curve takes the form $V^{\lambda_1}\oplus V^{\lambda_2}$ (where $V^{\lambda_2}=\overline{V^{\lambda_1}}$).  Then the following Hodge structures appear:
\begin{itemize}[leftmargin=0.4cm]
    \item $U=$ discrete series.  The only non-vanishing cohomology is $H^2(\mathfrak{g},K;U\otimes V^{n\lambda_1})$. The Hodge types are $(p,q)=$ $(n+2,0)$,$(n+1,1)$, and $(n,2)$.
    \item $U=$ the representation $J^n_{1,0}$. The non-vanishing cohomologies are $H^1$ and $H^3$. Their Hodge types are $(n+1,0)$, and $(n+2,1)$.
\end{itemize}

\noindent If $V^{\lambda_1}$ is twisted as for the local systems associated with $K3$ surfaces (resp. Calabi-Yau manifolds), just add $(1,0)$ (resp. $(2,0)$) to the Hodge numbers.  Finally, in \emph{all} cases, viewing $V^{n\lambda_2}$ as $\overline{V^{n\lambda_1}}$ as a Hodge representation, apply complex conjugation $(a,b)\mapsto (b,a)$ to the Hodge types to get the correct ones for $V^{n\lambda_2}$. In summary:

\begin{proposition}
In these three geometric cases, we do not get Hodge classes in the cohomology of local systems of Picard curves or Calabi-Yau threefolds.  In the case of $K3$ surfaces, we get Hodge classes of type $(2,2)$ in $H^2(\mathfrak{g},K;U\otimes V^{\lambda_1})$; similarly, for the universal abelian surfaces ($=$Jacobians of Picard curves) over the 2-ball, we get Hodge classes of type $(2,2)$ in $H^2(\mathfrak{g},K;U\otimes V^{2\lambda_1})$.
\end{proposition}


\subsubsection{$SU(3,1)$}

The representations of $U(3)$ are similar to the case of $U(2)$. We consider the surjection $SU(3)\times U(1)\to U(3)$ with kernel $A=\langle \omega\rangle\cong\mathbb{Z}/3\mathbb{Z}$ where $\omega$ is a primitive third root of unity. Any irreducible representation $(\pi,V)$ of $SU(3)$ defines, by restriction, a character $\chi_{\pi}$ of $A$, which must be $\omega\mapsto \omega^{i(\pi)}(i(\pi)=0,1,2)$. A pair $(\pi,n)$ with $n\in\mathbb{Z}$ and $n\equiv i(\pi)\mod 3$ defines an irreducible representation of $U(3)$. And all irreducible representations of $U(3)$ are of this form. The representation $\tau_1$ is $(\mathrm{st},1)$ and $\tau_1^*$ is $(\mathrm{st}^*,-1)$.

We have three fundamental representations. The standard representation $V^{\lambda_1}=W$, the wedge product $V^{\lambda_2}=\wedge^2W$, and $V^{\lambda_3}=\wedge^3W=W^*$. All other finite-dimensional representations are subrepresentations of tensor products of fundamental representations.


When $V$ is the trivial representation, the cohomological representations $U$ are:

\begin{itemize}[leftmargin=0.4cm]
\item four discrete series $D_i$. The only non-vanishing cohomology is $H^3(\mathfrak{g},K,U)=\mathbb{C}$. Their Hodge types are $(3,0)$, $(2,1)$, $(1,2)$, and $(0,3)$.
\item non-tempered representations. 
\begin{itemize}[leftmargin=0.6cm]
    \item $J_{0,0}=\mathbb{C}$,
    \item $J_{1,0}$ and $J_{0,1}$, whose only nonvanishing cohomology groups are $H^1(\mathfrak{g},K,U)=H^3(\mathfrak{g},K,U)=H^5(\mathfrak{g},K,U)=\mathbb{C}$. Their Hodge types are $(1,0)$ ($(2,1)$, and $(3,2)$) and $(0,1)$ ($(1,2)$, and $(2,3)$).
    \item $J_{2,0}$, $J_{1,1}$, $J_{0,2}$. The nonvanishing cohomology groups are $H^2(\mathfrak{g},K,U)=H^4(\mathfrak{g},K,U)=\mathbb{C}$. Their Hodge types are $(2,0)$, $(1,1)$, $(0,2)$ in degree two and $(3,1)$, $(2,2)$, $(1,3)$ in degree four.
\end{itemize}
\end{itemize}

For regular representations, only discrete series are cohomological. For singular representations, the twists of non-tempered representations are again cohomological, but we need to compute their minimal $K$-type to determine which representations are cohomological with respect to a finite-dimensional representation.  The next proposition treats the most singular cases:
\begin{itemize}[leftmargin=0.4cm]
\item If $\lambda=n\lambda_1$, a twist $J^{n}_{1,0}$ or $J^{n}_{2,0}$ is also cohomological;
\item If $\lambda=n\lambda_2$, a twist of $J^{n}_{1,1}$ is also cohomological;
\item If $\lambda=n\lambda_3$, a twist of $J^{n}_{0,1}$ or $J^{n}_{0,2}$ is also cohomological.
\end{itemize}

\begin{proposition}\label{p3.9}
For these choices of $\lambda$, with Hodge types as in tensor powers of $\mathcal{H}^1$ of Picard curves, the only case to get Hodge classes in $H^*(\mathfrak{g},K,U\otimes V^{\lambda})$ is when $\lambda =n\lambda_2$ and $U$ is a twisting of $J_{1,1}$.
\end{proposition}

\subsubsection{$Sp(4)$}

As a Lie group, the symplectic group $Sp_4(\mathbb{R})$ is defined to be the group
$$\{X\in M_{4\times 4}|X^tJX=J\}$$
where $J$ is the standard symplectic form $\begin{pmatrix}
0 & I_2 \\
-I_2 & 0
\end{pmatrix}$. The maximal torus of $Sp_4$ is the diagonal matrices $\mathrm{diag}(x,y,x^{-1},y^{-1})$. The maximal compact subgroup is $K=U(2)$. Let $X'=A+iB \in U(2)$ where $A,B$ are real matrices, the associated element in $Sp_4$ is $X=\begin{pmatrix}
A & B \\
-B & A
\end{pmatrix}$.

There are two fundamental weights $\lambda_1$ and $\lambda_2$. The corresponding representation $V=V^{\lambda_1}$ is the standard representation $W$. $V^{\lambda_2}$ is the kernel of the natural contraction map $\wedge^2W\to \mathbb{C}$ induced by the symplectic form over $W$. So we have $\wedge^2V^{\lambda_1}=V^{\lambda_2}\oplus \mathbb{C}$. We change the indexing by writing $V^{\lambda_1}$ as $V_{1,0}$, $V^{\lambda_2}$ as $V_{1,1}$, so that (more generally) the irreducible representations $V_{a,b}=V^{(a-b)\lambda_1+b\lambda_2}$ are parameterized by pairs of integers $\{(a,b)|a\geq b\geq0\}$. 


Now set 
\begin{align*}
\Xi_1&=\{(l_1,l_2)\in \mathbb{Z}^2| l_1>l_2>0\},&
\Xi_2&=\{(l_1,l_2)\in \mathbb{Z}^2| l_1>-l_2>0\},\\ 
\Xi_3&=\{(l_1,l_2)\in \mathbb{Z}^2| -l_2>l_1>0\},& 
\Xi_4&=\{(l_1,l_2)\in \mathbb{Z}^2| -l_2>-l_1>0\}.
\end{align*}
Then there is a one-to-one correspondence between the union of the four sets and the set of unitary equivalence classes of discrete series representations of $Sp_4(\mathbb{R})$. The integral point $(l_1,l_2)$ is called the \textit{Harish-Chandra parameter}.

\begin{enumerate}[leftmargin=0.8cm]
\item $D(l_1,l_2)$ denotes the holomorphic discrete series representation with the minimal $K$-type $\det^{l_2+2}\otimes \Sym^{l_1-l_2-1}$ if $(l_1,l_2)\in \Xi_1$.
\item $D(l_1,l_2)$ denotes the large discrete series representation with the minimal $K$-type $\det^{l_2}\otimes \Sym^{l_1-l_2+1}$ if $(l_1,l_2)\in \Xi_2$.
\item $D(l_1,l_2)$ denotes the large discrete series representation with the minimal $K$-type $\det^{l_2-1}\otimes \Sym^{l_1-l_2+1}$ if $(l_1,l_2)\in \Xi_3$.
\item $D(l_1,l_2)$ denotes the anti-holomorphic discrete series representation with the minimal $K$-type $\det^{l_2-1}\otimes \Sym^{l_1-l_2-1}$ if $(l_1,l_2)\in \Xi_4$.
\end{enumerate}
If $(l_1,l_2)\in\Xi_1$, the four discrete series representations $$\{D(l_1,l_2), D(l_1,-l_2), D(l_2,-l_1),D(-l_2,-l_1)\}$$ form a single $L$-packet, and they have the same infinitesimal character. $D(l_1,l_2)$ is \textit{holomorphic}, $D(-l_2,-l_1)$ is \textit{anti-holomorphic}, and the other two discrete series representations are called \textit{large discrete series representations}.

\vspace{.3in}

Next consider the parabolic subgroups
$$P_1(\mathbb{R})=\left(\begin{pmatrix}
* & * & * &*\\
* & * & * &*\\
0&0&*&*\\
0&0&*&*
\end{pmatrix} \in Sp_4(\mathbb{R})\right)\text{and}  \ P_2(\mathbb{R})=\left(\begin{pmatrix}
* & * & *&*\\
0 & * & *&*\\
0&0&*&0\\
0&*&*&*
\end{pmatrix}\in Sp_4(\mathbb{R})\right), $$
and let $M_k(\mathbb{R})A_k(\mathbb{R})N_k(\mathbb{R})$ be the Langlands decomposition of $P_k(\mathbb{R})$ for $k=1$ or $2$. Hence, we have
$M_1(\mathbb{R})\cong SL_2^{\pm}(\mathbb{R})$, $A_1(\mathbb{R})\cong \mathbb{R}^+$, $N_1(\mathbb{R})\cong \mathbb{R}^3$, $M_2(\mathbb{R})\cong SL_2(\mathbb{R})\times (\mathbb{Z}/2\mathbb{Z})$, $A_2(\mathbb{R})\cong\mathbb{R}^+$, and $N_2(\mathbb{R})\cong \mathbb{R} \ltimes \mathbb{R}^2$.

Let $D_k^+$ (resp. $D_k^-$) denote the holomorphic (resp. anti-holomorphic) discrete series of $SL_2(\mathbb{R})$ which has the minimal $K$-type $\begin{pmatrix}
\cos\theta & \sin\theta \\
-\sin\theta & \cos\theta
\end{pmatrix}\mapsto e^{i(k+1)\theta}$ (resp. $e^{-i(k+1)\theta}$). The representation $D_k:=\Ind_{SL_2(\mathbb{R})}^{SL_2^{\pm}(\mathbb{R})}D_k^+\cong \Ind_{SL_2(\mathbb{R})}^{SL_2^{\pm}(\mathbb{R})} D_k^-$ is irreducible, and its restriction to $SL_2(\mathbb{R})$ is $D_k^+\oplus D_k^-$. The quasi-character $\nu_1$ on $\mathbb{R}^+$ is defined by $\nu_1(a)=a$. Let $\sgn$ denote the non-trivial character on $\mathbb{Z}/2\mathbb{Z}$.

For each integer $k> 2$, we denote by $\sigma_k$ the Langlands quotient of the induced representation
$$\Ind^{Sp_4(\mathbb{R})}_{M_1(\mathbb{R})A_1(\mathbb{R})N_1(\mathbb{R})}(D_{2k-1}\otimes \nu_1\otimes 1).$$
Similarly, for each integer $l>1$, let $\omega_l^{\pm}$ denote the Langlands quotient of the induced representation
$$\Ind^{Sp_4(\mathbb{R})}_{M_2(\mathbb{R})A_2(\mathbb{R})N_2(\mathbb{R})}((D^{\pm}_{l}\otimes \sgn)\otimes \nu_1\otimes 1).$$
The representations $\sigma_k$ and $\omega_l^{\pm} $ are unitarizable and cohomological, but non-tempered.

Now let $V=V_{a,b}$ be an irreducible finite-dimensional representation. If $V$ is regular ($a>b>0$), only the discrete series representations are cohomological. If $a>b=0$, $\omega_{a+2}^{\pm}$ is also cohomological. If $a=b>0$, $\sigma_{a+3}$ is also cohomological. If $a=b=0$, all cohomological representations occur. 

So when $V=V_{0,0}$ is the trivial representation, we have the following cohomological representations and Hodge types:
\begin{itemize}[leftmargin=0.4cm]
\item four discrete series representations, with $H^3(\mathfrak{g},K,U)=\C$. Their Hodge types are $(3,0)$, $(2,1)$, $(1,2)$, and $(0,3)$. 
\item non-tempered representations $\omega_2^+$, $\sigma_3$, $\omega_2^-$: $H^2(\mathfrak{g},K,U)=H^4(\mathfrak{g},K,U)=\C$. Their Hodge types for $H^2$ are $(2,0)$, $(1,1)$, and $(0,2)$ (for $H^4$, add $(1,1)$).
\end{itemize}
More generally, we arrive at the following
\begin{proposition}\label{p3.10}
The only case to get Hodge classes in $H^*(\mathfrak{g},K,U\otimes V)$ is when $V$ is $V_{a,a}$ and $U$ is the non-tempered representation $\sigma_{a+3}$.  
\end{proposition}

\section{Some multiplicity formulas}

\subsection{Saito-Kurokawa lifting and non-tempered representations}

We first consider representations of $Sp(4,\R)$. Let $P(\R)$ be the Siegel parabolic subgroup, and $P(\R)=M(\R)A(\R)N(\R)$ the Langlands decomposition of $P$. Hence we have $M(\R)\cong SL_2^{\pm}(\R)$, $A(\R)\cong \R^+$, $N(\R)\cong \R^3$. 

Let $D_k^+$ (resp. $D_k^-$) denote the holomorphic (resp. anti-holomorphic) discrete series of $SL_2(\mathbb{R})$ which has the minimal $K$-type $\begin{pmatrix}
\cos\theta & \sin\theta \\
-\sin\theta & \cos\theta
\end{pmatrix}\mapsto e^{i(k+1)\theta}$ (resp. $e^{-i(k+1)\theta}$). The representation $D_k:=\Ind_{SL_2(\mathbb{R})}^{SL_2^{\pm}(\mathbb{R})}D_k^+\cong \Ind_{SL_2(\mathbb{R})}^{SL_2^{\pm}(\mathbb{R})} D_k^-$ is irreducible, and its restriction to $SL_2(\mathbb{R})$ is $D_k^+\oplus D_k^-$. The quasi-character $\nu_1$ on $\mathbb{R}^+$ is defined by $\nu_1(a)=a$. Let $k\geq 2$ be an integer. We denote by $\sigma_k$ the Langlands quotient of the induced representation
$$\mathrm{Ind}_{P(\R)}^{Sp(4,\R)}(D_{2k-1}\otimes \nu_1\otimes 1).$$
The representations $\sigma_k$ are unitarizable and cohomological. The  infinitesimal character is $(k-1,2-k)$ and the minimal $K$-type is $k-1,1-k$. Let $\sigma_{k}^-$ be the cohomological representation of $PGSp(4,\R)$ given in \cite{Sch1}. It is of course the same as a representation of $GSp(4,\R)$ with trivial central character. The restriction of $\sigma_k^-$ to $Sp(4,\R)$ is exactly  $\sigma_k$.

Consider the natural maps $i: Sp(4,\R)\to GSp(4,\R)$ and $p:GSp(4,\R)\to PGSp(4,\R)$. The kernel of the composition $\pi\circ i$ is $\pm I_4$. Now let  $\Gamma$ be a congruence subgroup of $Sp(4,\R)$ such that $-1\in \Gamma$. We are interested in the multiplicity of $\sigma_k$ inside $\mathcal{A}^2(\Gamma\backslash Sp(4,\R))$. Note that any paramodular subgroup of $Sp(4,\R)$ satisfies the above condition. The congruence subgroup $\Gamma$ can be identified with a congruence subgroup of $PGSp(4,\R)$, we also denote it by $\Gamma$.

($GSp(4,\R)$ has two componets. Let $GSp(4,\R)^+$ be the subgroup consisting of elements with positive multipliers. This is a product of $Sp(4,\R)$ and scalers. Let $I'$ be an element with multiplier $-1$. Any $\Gamma\subset Sp(4,\R)$ defines a subgroup of $GSp(4,\R)$ by taking products with scalers and $I'$. The group $PGSp(4,\R)$ has two components since scalers have positive multipliers. The group $\Gamma\subset PGSp(4,\R)$ is generated by the image of $\Gamma\subset Sp(4,\R)$ and $I'$.)

\begin{lemma} The multiplicity of $\sigma_k$ in $\mathcal{A}^2(\Gamma\backslash Sp(4,\R))$ is equal to the multiplicity of $\sigma_k^-$ in $\mathcal{A}^2(\Gamma\backslash PGSp_4(4,\R))$.
\end{lemma}

\begin{proof}
The two spaces  $\Gamma\backslash Sp(4,\R)$ and $\Gamma\backslash PGSp(4,\R)$ are actually the same.
\end{proof}

\subsubsection{Arthur's classification of automorphic representations of $GSp(4)$}

The multiplicity of $\sigma_k^-$ in $\mathcal{A}^2(\Gamma\backslash PGSp(4,\R))$ can be calculated in terms of automorphic forms. Let $\Gamma'$ be the corresponding open subgroup of $GSp(4,\A_f)$. Let $\pi$ be an automorphic representation of $GSp(4,\A)$ and $\pi=\pi_f\otimes \pi_{\infty}$ be its decomposition into nonarchimedean and archimedean components. Then 
$$\dim\Hom(\sigma_k^-,\mathcal{A}^2(\Gamma\backslash PGSp(4,\R)))=\sum\limits_{\pi_{\infty}=\sigma_k^-}\dim \pi_f^{\Gamma'}\mathrm{mult}(\pi,\mathcal{A}^2(GSp(4,\A))).$$

Arthur \cite{Art} classified the discrete spectrum of the group $GSp(4)$. They come in finite or infinite packets, of which there are six types. The general type consists of those representations that lift to cusp forms on $GL(4)$. The Yoshida type can be characterized as representations whose $L$-functions are of the form $L(s,\pi_1)L(s,\pi_2)$ with distinct cuspidal automorphic representations on $GL(2)$. At least conjecturally these two types consist of everywhere-tempered representations. Then there are three non-tempered types (Howe-Piatetski-Shapiro type, Saito-Kurakowa type, and Soudry type), associated with the three conjugacy classes of parabolic subgroups. Finally, there is a type consisting of one-dimensional representations. The automorphic representations of each type are parameterized by Arthur parameters, and their multiplicities are given explicitly in terms of Arthur parameters. See \cite{Sch3} for a more detailed description of the six types.

Now we want to study the multiplicity of $\sigma_k^{-}$ in $L^2_{\mathrm{disc}}(\Gamma\backslash PGSp_4(\mathbb{R}))$. Except for minimal $K$-type $(1,-1)$, the $\sigma_k^-$ can only appear in packets of Saito-Kurokawa (SK) type. They cannot appear in packets of general type or Yoshida type, since they are non-tempered. And they cannot appear in packets of Howe-Piatetski-Shapiro type or Soudry type, since these can be explicitly calculated (See \cite{Sch3}). Packets of SK-type are parametrized by pairs ($\mu$,$\sigma$), where $\mu$ is a cuspidal representation of $GL(2,\A)$ with trivial central character, and $\sigma$ is a quadratic Hecke character. All representations of SK-type are obtained via lifting from $GL(2)$ by Arthur's classification.

\subsubsection{Automorphic representations associated with elliptic modular forms}Let $f$ be a cuspidal new form for $\Gamma_0(N)$ with weight $k\geq 2$, we associate an automorphic representation $\pi_f$ of $GL_2(\A)$ as in \cite{Bu}. The representation $\pi_f$ is irreducible if $f$ is a new form, and $\pi_f$ is a restricted tensor product of irreducible representations $\pi_{f,v}$ over local groups $GL_2(\Q_{v})$. Given such an $f$, An explicit algorithm for computing local components is given in \cite{LW}.

If $v=\infty$, the representations of  $GL_2(\R)$ are well-known. In particular, $\pi_{f,\infty}$ is the unique discrete series subrepresentation of the representation constructed via unitary induction from the character
$$\begin{pmatrix}
    t_1&*\\
    & t_2
\end{pmatrix}\mapsto \left|\frac{t_1}{t_2}\right|^{k/2}\mathrm{sgn}(t_1)^k$$
of the Borel subgroup of $GL_2(\R)$.
 
Now let $p$ be a finite prime, the irreducible infinite-dimensional representations of $GL_2(\Q_p)$ fall into three classes: the principal series representations $\pi(\chi_1,\chi_2)$, the special representations $\mathrm{St}\otimes  \chi_1$, and the supercuspidal  representations $\mathrm{Ind}_K^G\tau$ (See \cite{BH}). Here, $\chi_1$ and $\chi_2$ are characters of $\Q_p^*$. The representation denoted by $\pi(\chi_1, \chi_2)$ is the principal series attached to the characters $\chi_1$ and $\chi_2$, defined whenever $\chi_1/\chi_2\neq |\cdot|^{\pm 1}$. The representation $\mathrm{St}$ is the Steinberg representation. The supercuspidal representations of $G$ are induced from certain finite-dimensional characters $\tau$ of compact-mod-center subgroups $K\subset G$.

\subsubsection{Some explicit multiplicity formulas}

Now let $p$ be a prime and we consider $\Gamma=\Gamma^{\mathrm{para}}(p)$. Let $\Gamma'$ be the corresponding open subgroup of $GSp(4,\A_f)$, then $\Gamma'_p=K(p)$, the local paramodular subgroup , and $\Gamma'_{p'}=GSp(4,\mathcal{O}_{p'})$ for $p'\neq p$. Therefore, we are looking for (discrete) automorphic representations $\pi$ with trivial central character, with archimedean component $\pi_\infty=\sigma_k^-$, with $p$-component $\pi_p$ such that $\pi_p$ admits non-zero $K(p)$-invariant vectors, and the other non-archimedean components unramified. The archimedean condition forces $\pi$ to be of SK-type.

Now we look at Table 2 in \cite{Sch2}. We see that $\mu_\infty$ must be a discrete series representation of weight $2k-2$ so that $\mu$ corresponds to a cuspidal newform f of this weight. Looking at the possibilities for $\pi_p$, we see that it must be of type IIb or Vb or VIc, as the others do not admit $K(p)$-invariant vectors. We also see that $\sigma$ must be unramified at every place, or there is at least one place where the representation does not have paramodular vectors. (Table A.12 in \cite{RS}) Thus $\sigma$ is in fact trivial, and $\mu$ has trivial central character. It follows that $f\in S_{2k-2}(\Gamma_0(N))$ for some $N$. In fact, to produce IIb at the place $p$ we must have $N=1$, and to produce Vb or VIc at $p$ we must have $N=p$. The next step is to determine which local new forms lift via Arthur's multiplicity formula. 
(The determination of the levels is because the local components of a new form can be explicitly computed.)

\begin{itemize}[leftmargin=0.8cm]
\item[IIb.]
Arthur's multiplicity formula says that in order for $\mu$ to lift, we must have $(-1)^n=\epsilon(1/2,\mu)$, where $n$ is the number of places where we do not have the base point in the local Arthur packet. Again from Table 2 we see $n=0$, so that we must have $\epsilon(1/2,\mu)=1$. The archimedean contribution to the epsilon-factor is $(-1)^{k-1}$. So we must have $\epsilon(1/2,\mu_p)=1$ if $k$ is odd, and $\epsilon(1/2,\mu_p)=-1$ if $k$ is even. We see the following: If $k$ is even, then there are no lifts producing IIb (at the place $p$), and if $k$ is odd, then the number of lifts with IIb is $\dim S_{2k-2}(SL(2,\Z))$. 
\vspace{0.1in}

\item[Vb.] Now consider lifts with a Vb component at $p$. From Table 2, $\mu_p$ must be a non-trivial twist of the Steinberg representation. But it must also be an unramified twist, or we won't have $K(p)$-invariant vectors. Thus $\mu_p=\xi\mathrm{St}_{GL(2)}$, where $\xi$ is the non-trivial, quadratic, unramified character of $\Q_p$. The local sign of such $\mu_p$ is $+1$. Thus, to satisfy Arthur's multiplicity formula, we must have $k$ odd. Hence, for $k$ odd, exactly the newforms in $S_{2k-2}(\Gamma_0(p))^+$ lift, where "$+$" indicates that the Atkin-Lehner eigenvalue at $p$ is $+1$. The number of automorphic representations as above with Vb at the place $p$ is therefore $\dim S_{2k-2}(\Gamma_0(p))^{new,+}$ for odd $k$. For even $k$ there are no such representations.\vspace{0.1in}

\item[VIc.]
Similarly, the number of automorphic representations as above with VIc at the place $p$ is $\dim S_{2k-2}(\Gamma_0(p))^{new,-}$ for even $k$ (while for odd $k$ there are no such representations). 
\end{itemize}

In summary, we have proved the following result.

\begin{proposition}\label{thB}
Let $\Gamma=\Gamma^{\mathrm{para}}(p)$ be a paramodular subgroup of prime level, then the multiplicity of $\sigma_k$ in $\mathcal{A}^2(\Gamma\backslash Sp(4,\R))$ is: 
\begin{itemize}[leftmargin=0.4cm]
\item $\dim S_{2k-2}(SL(2,\Z))+\dim S_{2k-2}(\Gamma_0(p))^{new,+}$ if $k$ is odd;
\item $\dim S_{2k-2}(\Gamma_0(p))^{new,-}$ if $k$ is even.
\end{itemize}
\end{proposition}

The dimension of elliptic modular forms can be found on \url{http://www.lmfdb.org/ModularForm/GL2/Q/holomorphic/}. This will be used in the next section.

\section{Geometric Applications}

Finally we turn to Hodge classes and instances of the Hodge conjecture for universal families over the locally symmetric varieties, first in the case of an arithmetic quotient of Siegel upper half space, then for ball quotients of dimensions 2 and 3.

\subsection{Families over a Siegel modular threefold}

\subsubsection{Local systems and their cohomology}

Let  $\mathfrak{H}_2$ be the Siegel upper half space of degree two. The group  $Sp(4,\R)$ acts naturally on $\mathfrak{H}_2$ via M\"{o}bius transformation.  The quotient $X(N):=\Gamma^{\text{para}}(N)\backslash\mathfrak{H}_2$ is a moduli space for complex abelian surfaces with polarization type $(1,N)$. More generally, we can take $\Gamma$ to be any arithmetic subgroup and write $X(\Gamma)$.  Let $\pi\colon\mathcal{A}\to X(\Gamma)$ be the universal family. The local system $R^1\pi_*\C$ is exactly the local system $\V_{1,0}$. All other local systems are sub-local systems of the higher direct images of the fiber product of the universal family.  There is also a universal family of genus $2$ curves $\mathcal{C}\to X(\Gamma)$; let $\bar{\mathcal{C}}$ denote a smooth compactification.

By Prop.~\ref{p3.10}, the only real Hodge classes in cohomologies of fiber powers $\mathcal{A}^k$ and $\mathcal{C}^k$ of the total spaces come from $$H^2(\mathfrak{sp_4},U(2);U\otimes V_{\R})\cong \R\cong H^4(\mathfrak{sp_4},U(2);U\otimes V_{\R})$$ when $V$ is $V_{a,a}$ (for some $a\leq k$ resp.~$\tfrac{k}{2}$) and $U$ is the non-tempered representation $\sigma_{a+3}$. Note that when $a=0$, these are just pullbacks of Hodge classes from $X(N)$, and this is all that occurs for the fourfold $\mathcal{C}$; for $\mathcal{A}$, $\mathbb{V}_{1,1}$ occurs in the relative $\mathcal{H}^2$. 

\begin{lemma}
These real Hodge classes are actually rational Hodge classes.
\end{lemma}
\begin{proof}
Writing $X(N)^*$ for the Baily-Borel compactification, Props.~\ref{p3.4} and \ref{p3.10} give 
$$IH^{2}(X(N)^*,\V_{a,a})\cong H^{2}(\mathfrak{sp}_4,U(2);\sigma_{a+3}\otimes V_{{a,a};\R})^{\oplus \text{mult.}(\Gamma^{\text{para}}(N),\sigma_{a+3})}$$ 

which is pure of type $(a+1,a+1)$ (with the LHS obviously defined over $\Q$). The point is that the other cohomological representations for $\V_{a,a}$ (which are discrete series) only contribute to $H^3$. The $H^4$ case follows by duality.
\end{proof}

\subsubsection{The universal curve}

Let $\V$ be an arbitrary local system. It turns out that the only case for which $H^3(\mathfrak{sp}_4,U(2), U\otimes V)\neq 0$ is when $U$ is a discrete series. However, these cohomology groups do not have real Hodge classes. Therefore, $IH^3(\mathcal{A}_2,\mathbb{V})^{2,2}=0$ for any local system $\mathbb{V}$. This leads to another proof of Arapura's result in \cite{Ar1}.

\begin{theorem}[Arapura]
The Hodge conjecture for the universal curve $\mathcal{C}$ over $X(\Gamma)$ (and thus for $\overline{\mathcal{C}}$) is true.
\end{theorem}

\begin{proof}
Let $\V=\V_{1,0}=R^1f_*\mathbb{C}$, then $H^3(U, \V_{1,0})^{2,2}=0$ since the intersection cohomology surjects onto the weight-4 part of the usual cohomology. (This is exactly the key vanishing theorem in Arapura's proof.)  Since $\overline{\mathcal{C}}\setminus \mathcal{C}$ is a 3-fold, we get the HC for $\overline{\mathcal{C}}$ by Prop.~\ref{p1}(iii).
\end{proof}

\subsubsection{Fiber products of the universal curve}
Let $\mathcal{C}^n$ be the $n$-fold fiber power of the universal family. All local systems $\V_{a,a}$ occur in the higher direct images of $\mathcal{C}^n$ for some $n$.  If $\Gamma=\Gamma^{\text{para}}(p)$, we need to find $\dim S_{2a+4}(SL(2,\Z))+\dim S_{2a+4}(\Gamma_0(p))^{new,+}$ if $a$ is even, and $\dim S_{2a+4}(\Gamma_0(p))^{new, -}$ if $a$ is odd.

\begin{theorem}\label{thC}
For $p=1,2,3,5$, the Hodge conjecture holds for $\mathcal{C}^2$ over $X(p)$.    
\end{theorem}

\begin{proof}
The only interesting higher direct image $\mathcal{H}^2_{\mathcal{C}^2/X(p)}$ has three components: the trivial local system, the adjoint local system $\V_{2,0}$, and the local system $\V_{1,1}$. Now by Prop.~\ref{p3.10}, $\V_{1,1}$ is the only possible non-trivial local system in $\mathcal{C}^2$ that could have Hodge classes. (The trivial ones lead to Hodge classes pulled back from the base; since this has dimension 3, those Hodge classes are algebraic.) If $\Gamma=Sp(4,\Z)$, then there are no Hodge classes. 
   
For $p=2,3,5$, just note that $\dim S_6(\Gamma_0(p))^{new,-}=0$ for $p=2,3,5$. 
\end{proof}

\begin{remark}
By the same calculation, we get a real Hodge class for $\mathcal{C}^2$ over $X(7)$ since $\dim S_6(\Gamma_0(p))^{new,-}=1$. It would be an interesting problem to determine if this is a rational Hodge class, and, if so, find a cycle representing it.
\end{remark}


\subsubsection{The universal abelian surface}
This also has the feature that the only local system supporting Hodge classes is the copy of $\V_{1,1}$ in $\mathcal{H}^2_{\mathcal{A}/X(p)}$.

\begin{corollary}
The HC holds for any smooth compactification of the universal abelian surface $\mathcal{A}$ over $X(p)$, for $p=1,2,3,5$.
\end{corollary}

\subsection{Families over 2-dimensional ball quotients}

\subsubsection{Local systems and their cohomology}

Let $\mathbb{B}^2=SU(2,1)/U(2)$ be the two-dimensional ball, the Hermitian symmetric domain associated with the group $SU(2,1)$. Let $X=\Gamma\backslash \mathbb{B}^2$ be an arithmetic quotient, and $\pi:\mathcal{C}\to X$ be the universal family of Picard curves.\footnote{The maximal choice of $\Gamma$ would be of the form $SU(2,1;\mathcal{O}_K)$, with $K=\Q(\omega)$, $\omega=e^{2\pi i/3}$. The curves here are of the form $y^3=P(x)$, with $P$ of degree $4$, see for example \cite{GK}.} Denoting the representation $V^{a\lambda_1+b\lambda_2}$ by $V_{a,b}$, the local system $R^1\pi_*\C$ is equal to $\V_{1,0}\oplus \V_{0,1}$. All other local systems are sub-local systems of the higher direct images of the fiber product of the universal family.

The only real Hodge classes are in $H^2(\mathfrak{g},K; U\otimes V)$ when $V=V_{a,a}$ and $U=U_{nh}$ is the non-holomorphic discrete series. 



\subsubsection{Picard curves}

We do not need to consider the universal Picard curves over 2-ball quotients since this is three-dimensional and the Hodge conjecture holds automatically.  On the other hand, to consider the fiber products $\mathcal{C}^k$, we would need to compute the multiplicity of the non-holomorphic discrete series.  We have not done this here, but (esp.~for $k=2$) it might produce nontrivial cases of the HC that can be solved concretely.

\subsubsection{K3-surfaces}
For $0\leq k\leq 6$, there exist $K3$ surfaces $S$ with an automorphism $\alpha_S$ of order three such that:
$$H^2(S,\Q)=T_S\oplus N_S, \;\;\;N_S:=H^2(S,\Q)^{\alpha_S}\cong \Q^{8+2k}, \;\;\;H^{2,0}\subset T_S\otimes \C.$$
The action of $\alpha_S^*$ defines a structure of $\Q(\omega)$-vector space on $T_S$. The moduli space of such $K3$ surfaces is a quotient of the $q$-ball (where $q=6-k$). Let $k=4$; then we have:

\begin{theorem}\label{tK3}
The Hodge conjecture holds for the family of $K3$ surfaces over $\Gamma \backslash \mathbb{B}^2$.
\end{theorem}
\begin{proof}
We do not get Hodge classes when $V$ is the (half-twisted) standard representation, so $H^2(X,\mathcal{H}^2)$ has no Hodge classes. But the K3 surfaces have no $H^1$ or $H^3$. 
\end{proof}

\noindent Since the universal $K3$ surface is a fourfold, the HC holds for any smooth compactification as well.


\subsubsection{Rohde's Calabi-Yau threefolds}

J.C.~Rohde \cite{Roh} constructed families of Calabi-Yau threefolds with $q=h^{2,1}=6-k$, for $0\leq k\leq 6$, which are parameterized by a $q$-dimensional ball quotients. The construction is as follows. Let $\omega\in\C$ be a primitive cube root of unity and consider the elliptic curve
$$E=\C/\Z+\omega\Z.$$
Then $E$ has a natural automorphism of order three $\alpha_E$ defined by $z\mapsto \omega z$. The automorphism $\alpha_E$ gives a decomposition of $H^1(E,\C)$ into eigenspaces with eigenvalues $\omega$ and $\bar{\omega}$:
$$H^1(E,\C)=H^{1,0}(E)_{\omega}\oplus H^{0,1}(E)_{\bar{\omega}}.$$

Now we take $S$ to be the $K3$ surface in the last subsection (so $k=4$, $q=2$). The weight three polarized rational Hodge structure of $\alpha=\alpha_S\otimes \alpha_E$-invariants in the tensor product $H^2(S,\Q)\otimes H^1(E,\Q)$ is then of CY-type. Rohde shows that it is isomorphic to the third cohomology group of a CY threefold $X_S$ which is a desingularization of the singular quotient variety $(S\times E)/\alpha$. Clearly, the moduli space of such $X_S$ is the same as the moduli space of $S$, a quotient of the $2$-ball.

The CY-threefold $X_S$ still has an automorphism $\alpha_{X_S}$ of order three which is induced by $\alpha_S\times \alpha_E$. We have
$$H^3(X_S,\C)\cong (T_{S,\bar{\omega}}\otimes H^{1,0}(E)_{\omega})\oplus (T_{S,\omega}\otimes H^{0,1}(E)_{\bar{\omega}}),$$
and 
$$H^{2,1}(X_S)\cong T^{2,0}_{S,\bar{\omega}}\otimes H^{1,0}(E)_{\xi}, \;\;\;\;\;\dim H^{2,1}(X_S)=q.$$
Consider the universal family $\mathcal{X}$ over the ball quotient $X=\Gamma\backslash \mathbb{B}_2$. The $H^1$ of CY-threefolds vanishes automatically. The middle cohomology $H^3$ has Hodge numbers $(1,2,2,1)$. As Hodge structures, they are just the double half-twists we considered in this paper. From the previous calculations, we know $H^k(X,\mathbb{V}^3)$ has no Hodge classes. Since the fiberwise $\mathcal{H}^{2j}$'s are copies of the trivial local system, any Hodge classes ``come from the base'' hence are algebraic.

\begin{proposition}
The HC holds for the (open) total spaces of these families of CY 3-folds.
\end{proposition}


\subsection{Curves over three-dimensional ball quotients}


We just consider the universal curve $\mathcal{C}$. The local system is associated with the standard representation.  By Prop.~\ref{p3.9}, there are no Hodge classes in $H^*(X,\tilde{\mathbb{V}}_{\R}^{\lambda_1})\otimes \C=H^*(X,\V^{\lambda_1}\oplus \V^{\lambda_3})$, so all Hodge classes come from the base.  Since the base has dimension three, these classes are known to be algebraic.  This proves

\begin{theorem}\label{thA}
Let $\mathcal{C}$ be the universal curve over an arithmetic quotient of $\mathbb{B}_3$. Then the Hodge conjecture holds for $\mathcal{C}$, hence for any compactification $\overline{\mathcal{C}}$.
\end{theorem}


\begin{thebibliography}{99}
\bibitem[Ar1]{Ar1}
D. Arapura, {\em Algebraic cycles on genus two modular fourfolds}, Algebra Numb. Theory 13 (2019), 211-225.

\bibitem[Ar2]{Ar2} D. Arapura, \emph{Hodge cycles and the Leray filtration}, Pacific J. Math. 319 (2022), no. 2, 233-258.

\bibitem[Art]{Art}
J. Arthur, {\em Automorphic Representations of $GSp(4)$}, JHU Press, Baltimore, 2004, 65-81.

\bibitem[BFNP]{BFNP}
P. Brosnan, H. Fang, Z. Nie, G. Pearlstein, \emph{Singularities of admissible normal functions}, Invent. Math. 117 (2009), 599-629.

\bibitem[BW]{BW}
A. Borel and N. Wallach. {\em Continuous cohomology, discrete subgroups, and representations of reductive groups}, vol. {\bf 67} of {\em Mathematical Surveys and Monographs}. American Mathematical Society, Providence, RI, second edition, 2000

\bibitem[Bu]{Bu}
D. Bump, ``Automorphic Forms and Representations'', Cambridge Univ. Press, Cambridge, 1997.

\bibitem[BH]{BH}
C. Bushnell and G. Henniart, {\em The Local Langlands conjecture for $GL(2)$},  {\em Grundelehren der Mathematischen Wissenschaften}, vol. {\bf 335}, Springer-Verlag, Berlin, 2006.

\bibitem[CM]{CM}
M.de Cataldo and L. Migliorini. {\em The decomposition theorem, perverse sheaves and the topology of algebraic maps}, Bulletin of the American Mathematical Society Volume {\bf 46}, No. {\bf 4}. 2009, 535-633.

\bibitem[De]{DM}
P. Deligne, \emph{Hodge cycles on abelian varieties} (notes by J.S. Milne), in Lecture Notes in Mathematics 900, Springer, Berlin, 1982, pp. 9-100.

\bibitem[DK]{DK}
I.V. Dolgachev and S. Kond\={o}. {\em Moduli spaces of $K3$ surfaces and complex ball quotients}, in {\em Arithmetic and geometry background of hypergeometric functions}, vol {\bf 260} of {Progr. Math.}, pp. 43-100. Birkh\"{a}user, Basel, 2007.

\bibitem[FC]{FC}
G. Faltings and C-L. Chai, ``Degeneration of abelian varieties'', Ergebnisse der Math. 22, Springer-Verlag, Berlin, 1990.

\bibitem[FL]{FL}
R. Friedman and R.Laza, {\em Semi-algebraic Horizontal subvarieties of Calabi--Yau type}, Duke Math. J. 162 (2013), 2077-2148.

\bibitem[GK]{GK}
P. Gallardo and M. Kerr, \emph{Algebraic and analytic compactifications of moduli spaces}, Notices of the AMS 69 (2022), no. 9, 1476-1485.

\bibitem[GG]{GG}
M. Green and P. Griffiths, \emph{Algebraic cycles and singularities of normal functions}, in ``Algebraic cycles and motives'', pp. 206-263, LMS Lect. Not. Ser. 343, Cambridge Univ. Press, Cambridge, 2007.

\bibitem[GGK1]{GGK1}
M. Green, P. Griffiths, and M. Kerr. {\em Mumford-Tate groups and domains: their geometry and arithmetic}, Annals of Math Studies, no. {\bf 183}, Princeton University Press, 2012.

\bibitem[GGK2]{GGK2}
M. Green, P. Griffiths and M. Kerr, {\em Special values of automorphic cohomology classes}, Mem. Amer. Math. Soc. 231 (2014), no. 1088, vi+145pp

\bibitem[HZ]{HZ}
M.Harris and S. Zucker,
{\em Boundary cohomology of Shimura varieties III: Coherent cohomology on higher-rank boundary strata and applications to Hodge theory}, M\'em. Soc. Math. Fr. (N.S.)(2001), no. 85, vi+116 pp.

\bibitem[Ho]{Hol}
R. Holzapfel. {\em The ball and some Hilbert problems}. Lectures in Mathematics ETH Z\"urich. Birkh\"{a}user Verlag, Basel, 1995. 

\bibitem[KK]{KK} R. Keast and M. Kerr, \emph{Normal functions over locally symmetric varieties}, SIGMA 14 (2018), 116, 18 pages.

\bibitem[KL]{KL} M. Kerr and R. Laza, \emph{Hodge theory of degenerations, I: consequences of the decomposition theorem} (with an appendix by M. Saito), Selecta Math. 27 (2021), Paper No. {\bf 71}, 48 pp.

\bibitem[KP]{KP} M. Kerr and G. Pearlstein, \emph{An exponential history of functions with logarithmic growth}, in ``Topology of Stratified Spaces'', MSRI Pub. 58, Cambridge Univ. Press, New York, 2011.

\bibitem[Lo]{Loo}
E. Looijenga,
{\em $L^2$-cohomology of locally symmetric varieties}, Compos. Math. 67 (1988), 3-20.

\bibitem[LW]{LW}
D. Loeffler and J. Weinstein, {\em On the computation of local components of a newform}. Mathematics of computation, vol. {\bf 81}, No. {\bf 278}, April 2012, 1179-1200.

\bibitem[Ma]{Ma}
E. Markman, \emph{The monodromy of generalized Kummer varieties and algebraic cycles on their intermediate Jacobians}, J. Eur. Math. Soc. 25 (2023), 231-321.

\bibitem[MSYZ]{MSYZ} S. M\"uller-Stach, M. Sheng, X. Ye, K. Zuo, \emph{On the cohomology groups of local systems over Hilbert modular varieties via Higgs bundles}, AJM {\bf 137} (2015), 1-35.

\bibitem[Mu]{Mu} K. Murty, \emph{Exceptional Hodge classes on certain abelian varieties}, Math. Ann. 268 (1984), no. 2, 197-206.

\bibitem[PSa]{PSa}	C. Peters and M. Saito, \emph{Lowest weights in cohomology of variations of Hodge structure}, Nagoya Math. J. 206 (2012), 1-24.

\bibitem[PSt]{PSt} C. Peters and J. Steenbrink, ``Mixed Hodge structures'', Springer, 2007. 

\bibitem[Pe]{DP}
D. Petersen,
{\em Cohomology of local systems on the moduli of principally polarized abelian surfaces}, Pacific J. Math 275 (2015), 39-61.

\bibitem[RS]{RS}
B. Roberts and R. Schmidt, {\em Local newforms for $GSP(4)$}, vol. {\bf 1918} of {\em Lecture Notes in Mathematics}, Springer Berlin, 2007.

\bibitem[Roh]{Roh}
J. C. Rohde, {\em Cyclic coverings, Calabi-Yau manifolds and complex multiplication}, vol {\bf 1975} of {\em Lectures Notes in Mathematics}, Springer-Verlag, Berlin, 2009.

\bibitem[Sa]{Sa}
M. Saito, \emph{Decomposition theorem for proper K\"ahler morphisms}, Tohoku Math. J. 42 (1990), no. 2, 127-147.

\bibitem[SS]{SS}
L. Saper and M. Stern,
{\em $L_2$-cohomology of arithmetic varieties}, Ann. Math. 132 (1990), 1-69.

\bibitem[Sch1]{Sch1}
R. Schmidt. {\em The Saito-Kurokawa lifting and functoriality}. Amer. J. Math.{\bf 127}(2005), 209-240.

\bibitem[Sch2]{Sch2}
R.  Schmidt, {\em Paramodular forms in CAP representations of $GSp(4)$}
Acta Arith. {\bf 194} (2020), 319-340.

\bibitem[Sch3]{Sch3}
R.Schmidt, {\em Packet structure and paramodular forms.}
Trans. Amer. Math. Soc. {\bf 370} (2018), 3085-3112.

\bibitem[Sc]{Sc}
C. Schoen, \emph{Hodge classes on self-products of a variety with an automorphism}, Compos. Math. 65 (1988), 3-32.

\bibitem[VZ]{VZ}
D. Vogan and G. Zuckerman, {\em Unitary representations with non-zero cohomology}, Comp. Math. {\bf 53}(1984), 51-90.

\bibitem[Vo]{Vo}
C. Voisin, \emph{Hodge loci and absolute Hodge classes}, Compos. Math 143 (2007), 945-958.

\bibitem[Zu]{Z}
S. Zucker,
{\em The Hodge structures on the intersection homology of varieties with isolated singularities}, Duke Math. J. 55 (1987), 603-616.


\end{thebibliography}
\end{document}